\documentclass[11pt,a4paper]{article}
\usepackage[cp1251]{inputenc}
\usepackage{amsmath,amsthm}
\usepackage{amsfonts,amstext,amssymb,verbatim,epsfig}

\def\R{{\mathbb{R}}}

\def\Z{{\mathbb{Z}}}
\newcommand{\PP}{\mathbb{P}}
\newcommand{\TT}{\hat{\tau}}

\newcommand{\EE}{\mathbb{E}}
\newcommand{\E}{\mathbb{E}}
\renewcommand{\P}{\mathbb{P}}

\theoremstyle{usual}
\newtheorem{theorem}{Theorem}[section]
\newtheorem{corollary}{Corollary}[section]
\newtheorem{lemma}{Lemma}[section]
\newtheorem{proposition}{Proposition}[section]

\newtheoremstyle{likedef}
  {}%
  {}%
  {}%
  {\parindent}%
  {\bfseries}%
  {.}%
  {.5em}%
  {}%

\theoremstyle{likedef}

\newtheorem{remark}{Remark}

\numberwithin{equation}{section}

\begin{document}

\title{Outlets of 2D invasion percolation and multiple-armed incipient infinite clusters}

\author{Michael Damron\thanks{Mathematics Department, Princeton University, Fine Hall, Washington Rd., Princeton, NJ 08544. Email: mdamron@math.princeton.edu;
Research funded by an NSF Postdoctoral Fellowship}
\and
Art\"{e}m Sapozhnikov\thanks{EURANDOM, P.O. Box 513, 5600 MB Eindhoven, The Netherlands. Email: sapozhnikov@eurandom.tue.nl;
Research partially supported by the Netherlands Organisation for Scientific Research (NWO) under grant number 613.000.429.}
}
\date{November 2009}
\maketitle


\begin{abstract}
We study invasion percolation in two dimensions, focusing on properties of the outlets of the invasion and their relation to critical percolation and to incipient infinite clusters (IIC's).  First we compute the exact decay rate of the distribution of both the weight of the $k^{th}$ outlet and the volume of the $k^{th}$ pond.  Next we prove bounds for all moments of the distribution of the number of outlets in an annulus.  This result leads to almost sure bounds for the number of outlets in a box $B(2^n)$ and for the decay rate of the weight of the $k^{th}$ outlet to $p_c$.  We then prove existence of multiple-armed IIC measures for any number of arms and for any color sequence which is alternating or monochromatic.  We use these measures to study the invaded region near outlets and near edges in the invasion backbone far from the origin.
\end{abstract}

\section{Introduction}\label{secIntroduction}

\subsection{The model}\label{model}

Invasion percolation is a stochastic growth model both introduced and numerically studied independently by \cite{Chandler} and \cite{Lenormard}.
Let $G = (V,E)$ be an infinite connected graph in which a distinguished vertex, the origin, is chosen.
Let $(\tau_e)_{e\in E}$ be independent random variables, uniformly distributed on $[0,1]$.  The
{\it invasion percolation cluster} (IPC) of the origin on $G$ is defined as
the limit of an increasing sequence $(G_n)$ of connected subgraphs of $G$ as follows.
For an arbitrary subgraph $G' = (V',E')$ of $G$, we define the outer edge boundary of $G'$ as
\begin{eqnarray*}
\Delta G'=\{e=\langle x,y\rangle\in E\ :\ e\notin
E'\mathrm{,\ but\ }x\in V'\mathrm{\ or\ }y\in V'\}.
\end{eqnarray*}
We define $G_0$ to be the origin.
Once the graph $G_i=(V_i,E_i)$ is defined, we select the edge $e_{i+1}$ that minimizes $\tau$ on $\Delta G_i$.
We take $E_{i+1}=E_i\cup\{e_{i+1}\}$ and let $G_{i+1}$ be the graph induced by the edge set $E_{i+1}$.
The graph $G_i$ is called the \textit{invaded region} at time $i$.
Let $E_\infty = \cup_{i=0}^\infty E_i$ and $V_\infty = \cup_{i=0}^\infty V_i$.  Finally, define the IPC 
\[{\cal S} = (V_\infty, E_\infty).\]

In this paper, we study invasion percolation on two-dimensional lattices; however, for simplicity {\it we restrict ourselves hereafter to the square lattice} ${\Z}^2$ and denote by ${\E}^2$ the set of nearest-neighbour edges.  The results of this paper still hold for lattices which are invariant under reflection in one of the coordinate axes and under rotation around the origin by some angle. In particular, this includes the triangular and honeycomb lattices.

We define Bernoulli percolation using the random variables $\tau_e$ to make a coupling with the invasion immediate.
For any $p\in[0,1]$ we say that an edge $e\in{\mathbb E}^2$ is $p$-open if $\tau_e< p$ and $p$-closed otherwise.
It is obvious that the resulting random graph of $p$-open edges has the same distribution as
the one obtained by declaring each edge of ${\mathbb E}^2$ open with probability $p$ and closed with probability $1-p$,
independently of the state of all other edges.
The percolation probability $\theta(p)$ is the probability that the origin is in the infinite cluster of $p$-open edges.
There is a critical probability $p_c = \inf\{p~:~\theta(p)>0\}\in(0,1)$.
For general background on Bernoulli percolation we refer the reader to \cite{Grimmett}.

The first mathematically rigorous study of invasion percolation appeared in \cite{Newman}.
In particular, the first relations between invasion percolation and critical Bernoulli percolation were observed.  It was shown that, for any $p>p_c$, the invasion on $(\Z^d,\E^d)$ intersects the infinite $p$-open cluster with probability one.
In the case $d=2$ this immediately follows from the Russo-Seymour-Welsh theorem (see Section~11.7 in \cite{Grimmett}).
This result has been extended to much more general graphs in \cite{HPS}.
Furthermore, the definition of the invasion mechanism implies that if the invasion reaches the $p$-open infinite cluster for some $p$,
it will never leave this cluster.
Combining these facts yields that if $e_i$ is the edge added at step $i$ then $\limsup_{i\to\infty}\tau_{e_i}=p_c$.
It is well-known that for Bernoulli percolation on $(\Z^2,\E^2)$, the percolation probability at $p_c$ is $0$.
This implies that, for infinitely many values of $i$, the weight $\tau_{e_i}>p_c$.
The last two results give that $\TT_1=\max\{\tau_e~:~e\in E_\infty\}$ exists and is greater than $p_c$.
The above maximum is attained at an edge which we shall call $\hat e_1$.
Suppose that $\hat{e}_1$ is invaded at step $i_1$, i.e. $\hat e_1 = e_{i_1}$.
Following the terminology of \cite{Newman-Stein}, we call the graph $G_{i_1-1}$ the {\it first pond} of the invasion, denoting it by the symbol $\hat V_1$, and we call the edge $\hat{e}_1$ the {\it first outlet}.
The second pond of the invasion is defined similarly.
Note that a simple extension of the above argument implies that $\TT_2=\max\{\tau_{e_i}~:~e_i\in E_\infty,i>i_1\}$ exists and is greater than $p_c$.
If we assume that $\TT_2$ is taken on the edge $\hat{e}_2$ at step $i_2$,
we call the graph $G_{i_2-1}\setminus G_{i_1-1}$ the {\it second pond} of the invasion, and we denote it $\hat V_2$.
The edge $\hat e_2$ is called the {\it second outlet}.
The further ponds $\hat V_k$ and outlets $\hat e_k$ are defined analogously.
For a hydrological interpretation of the ponds we refer the reader to \cite{pond}.

Various connections between the invasion percolation and the critical Bernoulli percolation have been established
in \cite{pond}, \cite{Newman}, \cite{DSV}, \cite{Jarai}, \cite{Wilkinson}, and \cite{Zhang},
using both heuristics and rigorous arguments.
In the remainder of this section we review results concerning these relations.
Afterward, we will state the main results of the paper.

In \cite{Newman},
 it was shown that the empirical distribution of the $\tau$ value of the invaded edges converges to the uniform distribution on $[0,p_c]$.
The authors also showed that the invaded region has zero volume fraction,
given that there is no percolation at criticality, and that it has surface to volume ratio $(1-p_c)/p_c$.
This corresponds to the asymptotic surface to volume ratio for large critical clusters (see \cite{Newman-Schulman} and \cite{Kunz-Souillard}).
The above results indicate that a large proportion of the edges in the IPC belongs to big $p_c$-open clusters.

Very large $p_c$-open clusters are usually referred to as incipient infinite clusters (IIC).
The first mathematical definition of the IIC of the origin in $\Z^2$ was given by Kesten in \cite{KestenIIC}.
We give this definition in Section~\ref{secIICs}.
Relations between the IIC and the IPC in two dimensions were observed in \cite{Jarai,Zhang}.
The scaling of the moments of the number of invaded sites in a box was obtained there, and this turned out to be
the same as the scaling of the corresponding moments for the IIC.
Moreover, it was shown in \cite{Jarai} that far away from the origin the IPC locally coincides with the IIC.
(We give the precise statement in Theorem~\ref{thm1IIC}.)
However, globally the IPC and the IIC are very much different.
Their laws are mutually singular \cite{DSV}.

The diameter, the volume, and the $r$-point function of the first pond of the invasion were studied in \cite{pond,BPSV,DSV}.
It was shown that the decay rates of their distributions coincide respectively with the decay rates of the distributions of
the diameter, the volume, and the $r$-point function of the critical cluster of the origin in Bernoulli percolation.
The diameter of the $k$-th pond for $k\geq 2$ was studied in \cite{DSV}.
It appears that the tail of the distribution of the diameter of the $k$-th pond scales differently for different values of $k$.

There is a rather complete understanding of the IPC on a rooted regular tree.
In \cite{tree} the scaling behaviour of the $r$-point function and the distribution of the volume of the invaded region at and below a given height were explicitly computed. These scalings coincide with the corresponding scalings for the IIC.
Local similarity between the invaded region far away from the root and the IIC has been shown.
It is also true here that the laws of the IPC and the IIC are mutually singular.

In this paper we study the sequence of outlets $(\hat e_k)$ and the sequence of their weights $(\hat\tau_k)$.
In Theorem~\ref{thmTau} we give the asymptotic behaviour for the distribution of $\hat\tau_k$ for any fixed $k$.
For $k>1$, we compute the exact decay rate of the distribution of the size of the $k$-th pond in Theorem~\ref{thmVolumes}.
This result can be also seen as a statement about the sequence of steps $i_k$ at which $\hat e_k$ are invaded.  In Theorem~\ref{outletmomentthm}, we find uniform bounds on all moments of the number of outlets in an annulus.  We use this result in Theorem~\ref{outletboxthm} to derive almost sure bounds on the number of outlets in a box $B(2^n)$.  An important consequence of Theorem~\ref{outletboxthm} is Corollary~\ref{outletsascor}; it states almost sure bounds on the difference $(\hat\tau_k - p_c)$ and on the radii of the ponds.  

In Theorem~\ref{thmIIC} we prove the existence of an IIC with several infinite $p_c$-open and $p_c$-closed paths from a neighbourhood of the origin.
We then show in Theorem~\ref{thm2IIC} and Theorem~\ref{thm4IIC} that the local description of the invaded region near the backbone of the IPC far away
from the origin is given by the IIC with two infinite $p_c$-open paths, and the local description of the invaded region near an outlet of the IPC
far away from the origin is given by the IIC with two infinite $p_c$-open paths and two infinite $p_c$-closed paths so that these paths alternate.  Last, in Theorem~\ref{thmpivotal4IIC}, we see that the same description used in Theorem~\ref{thm4IIC} applies in the setting of critical Bernoulli percolation to the region near an edge which is pivotal for a left-right crossing of a large box, given that this edge is sufficiently far from the boundary of the box.

\subsection{Notation}\label{secNotation}
\noindent
In this section we collect most of the notation and the definitions used in the paper.

For $a\in\R$, we write $|a|$ for the absolute value of $a$, and,
for a site $x = (x_1,x_2)\in\Z^2$, we write $|x|$ for $\max(|x_1|,|x_2|)$.
For $n>0$ and $x\in\Z^2$,
let $B(x,n) = \{y\in\Z^2~:~|y-x|\leq n\}$ and $\partial B(x,n) = \{y\in\Z^2~:~|y-x|=n\}$.
We write $B(n)$ for $B(0,n)$ and $\partial B(n)$ for $\partial B(0,n)$.
For $m<n$ and $x\in\Z^2$, we define the annulus $Ann(x;m,n) = B(x,n)\setminus B(x,m)$.
We write $Ann(m,n)$ for $Ann(0;m,n)$.

We consider the square lattice $(\Z^2,{\mathbb E}^2)$, where ${\mathbb E}^2 = \{\langle x,y \rangle \in\Z^2\times\Z^2~:~|x-y|=1\}$.
Let $(\Z^2)^* = (1/2,1/2) + \Z^2$ and $({\mathbb E}^2)^* = (1/2,1/2) + {\mathbb E}^2$ be the vertices and the edges of the dual lattice.
For $x\in\Z^2$, we write $x^*$ for $x + (1/2,1/2)$.
For an edge $e\in{\mathbb E}^2$ we denote its endpoints (left respectively right or bottom respectively top) by $e_x, e_y\in \Z^2$.
The edge $e^* = \langle e_x + (1/2,1/2),e_y-(1/2,1/2) \rangle$ is called the \textit{dual edge} to $e$.
Its endpoints (bottom respectively top or left respectively right) are denoted by $e_x^*$ and $e_y^*$.
Note that, in general, $e_x^*$ and $e_y^*$ are not the same as $(e_x)^*$ and $(e_y)^*$.
For a subset ${\mathcal K}\subset \Z^2$, let ${\mathcal K}^* = (1/2,1/2) + {\mathcal K}$.
We say that an edge $e\in{\mathbb E}^2$ is in ${\mathcal K}\subset \Z^2$ if both its endpoints are in ${\mathcal K}$.  For any graph ${\cal G}$ we write $|{\cal G}|$ for the number of vertices in ${\cal G}$.

Let $(\tau_e)_{e\in{\mathbb E}^2}$ be independent random variables, uniformly distributed on $[0,1]$, indexed by edges.
We call $\tau_e$ the \textit{weight} of an edge $e$.
We define the weight of an edge $e^*$ as $\tau_{e^*} = \tau_e$.
We denote the underlying probability measure by ${\mathbb P}$ and the space of configurations by $([0,1]^{{\mathbb E}^2},{\mathcal F})$,
where ${\mathcal F}$ is the natural $\sigma$-field on $[0,1]^{{\mathbb E}^2}$.
We say that an edge $e$ is $p$-{\it open} if $\tau_e <p$ and $p$-{\it closed} if $\tau_e > p$.
An edge $e^*$ is $p$-open if $e$ is $p$-open, and it is $p$-closed if $e$ is $p$-closed.  Accordingly, for $p \in [0,1]$, we define the edge configuration $\omega_p \in \{0,1\}^{{\mathbb Z}^2}$ by

\[ \omega_p(e) = \begin{cases} 
1 & \tau_e \leq p \\
0 & \tau_e > p 
\end{cases} \]
and we make a similar definition for $\omega_p^*$, the dual edge configuration.
The event that two sets of sites ${\mathcal K}_1,{\mathcal K}_2\subset\Z^2$ are connected by a $p$-open path is denoted by
${\mathcal K}_1 \stackrel{p}\longleftrightarrow {\mathcal K}_2$, and
the event that two sets of sites ${\mathcal K}_1^*,{\mathcal K}_2^*\subset(\Z^2)^*$ are connected by a $p$-closed path in the dual lattice is denoted by
${\mathcal K}_1^* \stackrel{p^*}\longleftrightarrow {\mathcal K}_2^*$.
For any $n \geq 1$ and $p \in [0,1]$, we define the event
\[
\begin{array}{ccc}
B_{n,p}&=&\{\textrm{There is a }p\textrm{-closed circuit with radius at least } n \textrm{ around the origin in the dual lattice}\}.
\end{array}
\]

For $p\in[0,1]$, we consider a probability space $(\Omega_p,{\mathcal F}_p,{\mathbb P}_p)$, where
$\Omega_p = \{0,1\}^{{\mathbb E}^2}$, ${\mathcal F}_p$ is the $\sigma$-field generated by the finite-dimensional
cylinders of $\Omega_p$, and ${\mathbb P}_p$ is a product measure on $(\Omega_p,{\mathcal F}_p)$, defined as 
${\mathbb P}_p = \prod_{e\in{\mathbb E}^2}\mu_e$, where $\mu_e$ is the probability measure on vectors $(\omega_e)_{e\in{\mathbb E}^2}\in\Omega_p$ with $\mu_e(\omega_e = 1) = 1 - \mu_e(\omega_e = 0) = p$.
We say that an edge $e$ is {\it open} or {\it occupied} if $\omega_e = 1$, and $e$ is {\it closed} or {\it vacant} if $\omega_e = 0$.
We say that an edge $e^*$ is open or occupied if $e$ is open, and it is closed or vacant if $e$ is closed.
The event that two sets of sites ${\mathcal K}_1,{\mathcal K}_2\subset\Z^2$ are connected by an open path
is denoted by ${\mathcal K}_1 \leftrightarrow {\mathcal K}_2$, and
the event that two sets of sites ${\mathcal K}_1^*,{\mathcal K}_2^*\subset\Z^2$ are connected by a closed path in the dual lattice
is denoted by ${\mathcal K}_1^* \stackrel{*}\leftrightarrow {\mathcal K}_2^*$.  For any $n \geq 1$ and $p \in [0,1]$, let
\[ \pi_n={\PP}_{p_c}(0\leftrightarrow \partial B(n)) \textrm{ and } \pi(n,p) = {\PP}_p(0 \leftrightarrow \partial B(n)). \]
Also define the event
\[
\begin{array}{ccc}
B_n&=&\{\textrm{There is a closed circuit with radius at least }n \textrm{ around the origin in the dual lattice}\}.
\end{array}
\]
For any $k \geq 1$, let $\hat R_k$ be the radius of the union of the first $k$ ponds.  In other words,

\[ \hat R_k = \max \{|x|~:~ x \in \cup_{j=1}^k \hat V_k\}. \]
For two functions $g$ and $h$ from a set ${\mathcal X}$ to ${\mathbb R}$,
we write $g(z) \asymp h(z)$ to indicate that $g(z)/h(z)$ is bounded away from $0$ and $\infty$, uniformly in $z\in{\mathcal X}$.
Throughout this paper we write $\log$ for $\log_2$.
We also write ${\mathbb P}_{cr}$ for ${\mathbb P}_{p_c}$.
All the constants $(C_i)$ in the proofs are strictly positive and finite.
Their exact values may be different from proof to proof.

\subsection{Main results}
\noindent

\subsubsection{Weight of the $k^{th}$ outlet}
\noindent

Let $\hat\tau_k$ be the weight of the $k^{th}$ outlet, as defined in Section~\ref{model}.
\begin{theorem}\label{thmTau}
For any $k\geq 1$,
\begin{equation}\label{eqTau}
{\mathbb P}(\hat\tau_k < p)
\asymp
\left(\log L(p)\right)^{k-1} \theta(p),~~~~p>p_c,
\end{equation}
where the correlation length $L(p)$ is defined in Section~\ref{secCL}.
\end{theorem}
\begin{remark}
Note that the statement is trivial in the case $k=1$.
Indeed, it follows from the definition of the invasion that ${\mathbb P}(\hat\tau_1 < p) = \theta(p)$ for all $p$.
\end{remark}

\subsubsection{Volumes of the ponds}
\noindent

Recall the definition of the $k^{th}$ pond $\hat V_k$ from Section~\ref{model} and the definition of $\pi_n$ from Section~\ref{secNotation}.

\begin{theorem}\label{thmVolumes}
For any $k\geq 1$,
\begin{equation}\label{eqVolumes}
{\mathbb P}(|\hat V_k|\geq n^2\pi_n)
\asymp
(\log n)^{k-1} \pi_n,~~~~n\geq 2.
\end{equation}
In particular, 
\begin{equation}\label{eqVolumes1}
{\mathbb P}(|\hat V_k|\geq n)
\asymp
(\log n)^{k-1} {\mathbb P}_{cr}(|C(0)|\geq n),~~~~n\geq 2.
\end{equation}
\end{theorem}
\begin{remark}
The case $k=1$ is considered in \cite{pond}.
\end{remark}
\begin{remark}
The second set of inequalities follows from the first one using the relations
${\mathbb P}_{cr}(|C(0)|\geq n^2\pi_n) \asymp \pi_n$ (see \cite[Theorem 2]{pond}) and
$\log (n^2\pi_n) \asymp \log n$.
\end{remark}
\begin{remark}
Let $i_k$ be the index such that $e_{i_k}= \hat e_k$. Then $i_k$ is comparable to $|\hat V_1| + \ldots + |\hat V_k|$.
Therefore the statements (\ref{eqVolumes}) and (\ref{eqVolumes1}) hold with $|\hat V_k|$ replaced by $i_k$.
\end{remark}

\subsubsection{Almost sure bounds}
\noindent
For any $m<n$, let $O(m,n)$ be the number of outlets in $Ann(m,n)$, and let $O(n)$ be the number of outlets in $B(n)$.  We first give $n$-independent bounds on all moments of $O(n,2n)$.

\begin{theorem}
\label{outletmomentthm}
There exists $c_1>0$ such that for all $t,n \geq 1$,

\begin{equation}
\label{outletmomenteq1}
{\E}(O(n,2n)^t) \leq (c_1t)^{3t}.
\end{equation}
In particular, there exists $c_2, \lambda > 0$ such that for all $n$,

\begin{equation}
\label{outletmomenteq2}
{\E}(\exp(\lambda O(n,2n)^{1/3})) < c_2.
\end{equation}

\end{theorem}

Next we show almost sure bounds on the sequence of random variables $(O(2^n))_{n \geq 1}$.

\begin{theorem}
\label{outletboxthm}
There exists $c_3, c_4 >0$ such that with probability one, for all large $n$,

\begin{equation}
\label{outletboxeq}
c_3n \leq O(2^n) \leq c_4n.
\end{equation}
\end{theorem}

Theorem~\ref{outletboxthm} implies related bounds on the convergence rate of the weights $\hat\tau_k$ to $p_c$ and on the growth of the radii  $(\hat R_k)_{k \geq 1}$, defined in Section~\ref{secNotation}.

\begin{corollary}
\label{outletsascor}

\begin{enumerate}
\item There exists $c_5$ and $c_6$ with $1<c_5,c_6<\infty$ such that with probability one, for all large $k$,

\begin{equation}
\label{outletsaseq1}
(c_5)^k \leq \hat R_k \leq (c_6)^k.
\end{equation}

\item There exists $c_7$ and $c_8$ with $0 < c_7, c_8 < 1$ such that with probability one, for all large $k$,

\begin{equation}
\label{outletsaseq2}
(c_7)^k \leq \hat\tau_k - p_c \leq (c_8)^k. 
\end{equation}

\end{enumerate}
\end{corollary}

\begin{remark}
Asymptotics of various ponds statistics as well as CLT-type and large
deviations results for deviations of those quantities away from their
limits are studied in \cite{Goodman} for invasion percolation on regular trees.
Not only do the results in \cite{Goodman} imply exponential almost sure bounds similar to
(\ref{outletsaseq1}) and (\ref{outletsaseq2}), they are very explicit. For instance, it is shown that $\lim_{k \to \infty}\frac{1}{k}\ln(\hat\tau_k-p_c)=-1$ and $\lim_{k \to \infty} \frac{1}{k} \ln(\hat R_k) = 1$ a.s.
\end{remark}

Our last theorem concerns ratios of successive terms of the sequence $(\hat\tau_k-p_c)_{k \geq 1}$.
\begin{theorem}
\label{outletdensitythm}
With probability one, the set

\[ \left\{ \frac{\hat\tau_{k+1}-p_c}{\hat\tau_k-p_c} ~:~ k \geq 1 \right\} \]
is a dense subset of $[0,1]$.
\end{theorem}

\subsubsection{Outlets and multiple-armed IICs}\label{secIICs}
\noindent

First we recall the definition of the incipient infinite cluster from \cite{KestenIIC}.
It is shown in \cite{KestenIIC} that the limit
\begin{equation*}
\nu(E) = \lim_{N\to\infty}{\mathbb P}_{cr}(E~|~0\leftrightarrow \partial B(N))
\end{equation*}
exists for any event $E$ that depends on the state of finitely many edges in ${\mathbb E}^2$.
The unique extension of $\nu$ to a probability measure on configurations of open and closed edges
exists.  Under this measure, the open cluster of the origin is a.s. infinite.  It is called the {\it incipient infinite cluster} (IIC).  In Theorem~\ref{thm1IIC} (\cite[Theorem~3]{Jarai}), a relation between IPC and IIC is given.

In this section we introduce multiple-armed IIC measures (Theorem~\ref{thmIIC}) and study
their relation to invasion percolation (Theorem~\ref{thm2IIC} and Theorem~\ref{thm4IIC}).

Let $\sigma\in\{\mbox{open},\mbox{closed}\}^{|\sigma|}$ be a finite vector (here $|\sigma|$ is the number of entries in $\sigma$).
For $l<n$ such that $|\partial B(l)| > |\sigma|$, we say that $B(l)$ is $\sigma$-connected to $\partial B(n)$, $B(l)\leftrightarrow_\sigma \partial B(n)$,
if there exist $|\sigma|$ disjoint paths between $B(l)$ and $\partial B(n)$ such that the number of open paths is given by
the number of entries `open' in $\sigma$, the number of closed paths is given by the number of `closed' in $\sigma$, and the relative
counterclockwise arrangement of these paths is given by $\sigma$. In the definition above we allow $n = \infty$, in this case we write
$B(l)\leftrightarrow_\sigma\infty$.

\begin{theorem}\label{thmIIC}
Suppose that $\sigma$ is alternating and let $l$ be such that $|\partial B(l)|\geq |\sigma|$.
For every cylinder event $E$, the limit
\begin{equation}\label{eqIIC}
\nu_\sigma(E) = \lim_{n\to\infty}{\mathbb P}_{cr}(E~|~B(l)\leftrightarrow_\sigma \partial B(n))
\end{equation}
exists. For the unique extension of $\nu_\sigma$ to a probability measure on the configurations of open and closed edges,
\[
\nu_\sigma(B(l)\leftrightarrow_\sigma \infty) = 1.
\]
We call the resulting measure $\nu_\sigma$ the $\sigma$-{\it incipient infinite cluster measure}.
\end{theorem}
\begin{remark}
Note that Kesten's IIC measure corresponds to the case $\sigma = \{\mbox{open}\}$. One can check that Kesten's original proof \cite{KestenIIC} also works for the case $\sigma = \{\mbox{open,open}\}$. We use this second IIC measure in Theorem~\ref{thm2IIC}.
\end{remark}

\begin{remark}
The proof we present of Theorem~\ref{thmIIC} can be easily modified to give the existence of IIC's for $\sigma$'s which either do not contain neighboring open paths or do not contain neighboring closed paths (here we take the first and last elements of $\sigma$ to be neighbors). In particular, it works for any 3-arm IIC and for monochromatic IIC's. In the case when there are neighboring open paths (but no neighbouring closed paths) one needs to change the proof by considering closed circuits with defects instead of open circuits.
\end{remark}

Let ${\cal S}$ be the invasion percolation cluster of the origin (see Section~\ref{model}) and define ${\cal O} = \{\hat e_k~:~k\geq 1\}$, the set of \textit{outlets} of the invasion.  Let ${\cal B}$ be the \textit{backbone} of the invasion, i.e., those vertices which are connected in ${\cal S}$ by two disjoint paths, one to the origin and one to $\infty$.  Recall the definitions of $e_x$ and $e_y$ from Section~\ref{secNotation}.  For any vertex $v$, define the shift operator $\theta_v$ on configurations $\omega$ so that for any edge $e$, $\theta_v(\omega)(e) = \omega(e-v),$ where $e-v = \langle e_x-v,e_y-v \rangle$.  For any event $E$, define

\[ \theta_vE = \{ \theta_v(\omega) ~:~ \omega \in E \}, \]
and if ${\cal K}$ is a set of edges in ${\E}^2$, define

\[ E_{\cal K} = \{ {\cal K} \subset {\cal S}\} \textrm{, } \theta_v {\cal K} = \{e \in {\EE}^2~:~e-v \in {\cal K} \} \textrm{, and } \theta_vE_{\cal K} = \{ \theta_v{\cal K} \subset {\cal S} \}. \]

If $e$ is an edge, we define the shift $\theta_e$ to be $\theta_{e_x}$.  The symbols $\theta_eE,$ $\theta_e{\cal K}$, and $\theta_eE_{\cal K}$ are defined similarly.  Let $E'_{\cal K}$ be the event that the set of edges ${\cal K}$ is contained in the cluster of the origin.  For any $p \in [0,1]$, note the definition of the edge configuration $\omega_p$ from Section~\ref{secNotation}. 

We recall \cite[Theorem~3]{Jarai}:

\begin{theorem}\label{thm1IIC}
Let $E$ be an event which depends on finitely many values $\omega_{p_c}(\cdot)$ and let ${\cal K} \subset {\E}^2$ be finite.

\[ \lim_{|v| \to \infty} {\P}(\theta_vE~|~ v \in {\cal S}) = \nu(E); \textrm{ and } \]
\[ \lim_{|v| \to \infty} {\P}(\theta_vE_{\cal K} ~|~ v \in {\cal S}) = \nu(E'_{\cal K}), \]
where the measure on the right is the IIC measure.
\end{theorem}
The above theorem states that asymptotically the distribution of invaded edges near $v$ is given by the IIC measure.

We are interested in the distribution of invaded edges near the backbone (Theorem~\ref{thm2IIC}) or near an outlet (Theorem~\ref{thm4IIC}).
While the analysis of the distribution of the invaded edges near the backbone is very similar to the proof of Theorem~\ref{thm1IIC},
the study of the distribution of the invaded edges near an outlet is more involved.

\begin{theorem}\label{thm2IIC}
Let $E$ be an event which depends on finitely many values $\omega_{p_c}(\cdot)$ and let ${\cal K} \subset {\E}^2$ be finite.

\[ \lim_{|v| \to \infty} {\P}(\theta_vE~|~ v \in {\cal B}) = \nu^{2,0}(E); \textrm{ and } \]
\[ \lim_{|v| \to \infty} {\P}(\theta_vE_{\cal K} ~|~ v \in {\cal B}) = \nu^{2,0}(E'_{\cal K}), \]
where the measure on the right is the (open,open)-IIC measure.
\end{theorem}

\begin{proof}
Similar to the proof of \cite[Theorem~3]{Jarai}.
\end{proof}

\begin{theorem}\label{thm4IIC}

Let $E$ be an event which depends on finitely many values $\omega_{p_c}(\cdot)$ $($but not on $\omega_{p_c}(e-e_x))$, and let ${\cal K}$ be a finite set of edges such that $e-e_x \notin {\cal K}$.

\[ \lim_{|e| \to \infty} {\P}(\theta_eE~|~ e \in {\cal O}) = \nu^{2,2}(E); \textrm{ and } \]
\[ \lim_{|e| \to \infty} {\P}(\theta_eE_{\cal K} ~|~ e \in {\cal O}) = \nu^{2,2}(E'_{\cal K}), \]
where the measure on the right is the $($open,closed,open,closed $)$-IIC measure.
\end{theorem}

The final theorem is inspired by \cite[Theorem 2]{Jarai2} and its proof is similar to (but much easier than) that of Theorem~\ref{thm4IIC} above, so we omit it.  Let ${\cal P}_n$ be the set of edges which are pivotal\footnote{For the definition of pivotality, we refer the reader to \cite{Grimmett}.} for the event

\[ \left\{ \begin{array}{c}
\textrm{There is an occupied path which connects the left and}\\
\textrm{right sides of }B(n) \textrm{ and which remains inside }B(n).
\end{array} \right\} \]

\begin{theorem}\label{thmpivotal4IIC}

Let $E$ be an event which depends on finitely many values $\omega_{p_c}(\cdot)$ and let ${\cal K} \subset {\E}^2$ be finite.  Let $h(n) \to \infty$ in such a way that $h(n) \leq n$.

\[ \lim_{\stackrel{n \to \infty}{|e| \leq n-h(n)}} {\mathbb P}_{cr}(\theta_eE ~|~ e \in {\cal P}_n) = \nu^{2,2}(E); \textrm{ and} \]
\[ \lim_{\stackrel{n \to \infty}{|e| \leq n-h(n)}} {\mathbb P}_{cr}(\theta_e E_{\cal K} ~|~ e \in {\cal P}_n) = \nu^{2,2}(E'_{\cal K}). \]
\end{theorem}
\subsection{Structure of the paper}
\noindent
We define the correlation length and state some of its properties in Section~\ref{secCL}.  We prove Theorems~\ref{thmTau} and \ref{thmVolumes} in Sections~\ref{tausec} and \ref{volumessec}, respectively.  The proofs of Theorems~\ref{outletmomentthm} - \ref{outletdensitythm} are in Section~\ref{almostsuresec}: the proof of Theorem~\ref{outletmomentthm} is in Section~\ref{outletmomentsec}; the proofs of Theorem~\ref{outletboxthm} and Corollary~\ref{outletsascor} are in Section~\ref{outletboxsec}; and the proof of Theorem~\ref{outletdensitythm} is in Section~\ref{densitysec}.  We prove Theorem~\ref{thmIIC} in Section~\ref{IICsec} and Theorem~\ref{thm4IIC} in Section~\ref{4IICsec}.  For the notation in Sections~\ref{tausec} - \ref{4IICsec} we refer the reader to Section~\ref{secNotation}.

\section{Correlation length and preliminary results}\label{secCL}
\noindent
In this section we define the correlation length that will play a crucial role in our proofs.
The correlation length was introduced in \cite{corrlength intro} and further studied in \cite{kesten}.

\subsection{Correlation length}
\noindent
For $m,n$ positive integers and $p\in (p_c,1]$ let
\begin{eqnarray*}
\sigma(n,m,p) = \PP_p(\mbox{there is an open horizontal crossing of }[0,n]\times[0,m]).
\end{eqnarray*}
Given $\varepsilon>0$, we define
\begin{eqnarray}
L(p,\varepsilon)=\min\{n~:~ \sigma(n,n,p)\geq 1-\varepsilon\}.
\end{eqnarray}
$L(p,\varepsilon)$ is called the finite-size scaling correlation length
and it is known that $L(p,\varepsilon)$ scales like the usual correlation length (see \cite{kesten}).
It was also shown in \cite{kesten} that the scaling of $L(p,\varepsilon)$ is independent of $\varepsilon$ given that it is small enough,
i.e. there exists $\varepsilon_0>0$ such that for all $0<\varepsilon_1,\varepsilon_2\leq \varepsilon_0$
we have $L(p,\varepsilon_1)\asymp L(p,\varepsilon_2)$.
For simplicity we will write $L(p)=L(p,\varepsilon_0)$ for the entire paper.
We also define $$p_n=\sup\{p~:~ L(p)>n\}.$$
It is easy to see that $L(p)\to\infty$ as $p\to p_c$ and $L(p)\to 0$ as $p\to 1$.
In particular, the probability $p_n$ is well-defined.
It is clear from the definitions of $L(p)$ and $p_n$ and from the RSW theorem that, for positive integers $k$ and $l$,
there exists $\delta_{k,l}>0$ such that, for any positive integer $n$ and for all $p\in [p_c,p_n]$,
\begin{equation*}
\PP_p(\mbox{there is an open horizontal crossing of }[0,kn]\times[0,ln]) >\delta_{k,l}
\end{equation*}
and
\begin{equation*}
\PP_p(\mbox{there is a closed horizontal dual crossing of }([0,kn]\times[0,ln])^*)>\delta_{k,l}.
\end{equation*}
By the FKG inequality and a standard gluing argument \cite[Section~11.7]{Grimmett} we get that, for positive integers $n$ and $k\geq 2$
and for all $p\in[p_c,p_n]$,
\begin{equation*}
\PP_p(Ann(n,kn)\mbox{ contains an open circuit around the origin}) >(\delta_{k,k-2})^4
\end{equation*}
and
\begin{equation*}
\PP_p(Ann(n,kn)^*\mbox{ contains a closed dual circuit around the origin}) >(\delta_{k,k-2})^4.
\end{equation*}

\subsection{Preliminary results}
\noindent
For any positive $l$ we define
$\log^{(0)}l = l$ and
$\log^{(j)}l=\log(\log^{(j-1)}l)$ for all $j\geq 1$,
as long as the right-hand side is well defined.
For $l>10$, let
\begin{equation}\label{eqDefLog*}
\log^*l=\min\{j> 0\ :\ \log^{(j)}l\mbox{ is well-defined and}\log^{(j)}l\leq 10\}.
\end{equation}

Our choice of the constant $10$ is quite arbitrary,
we could take any other large enough positive number instead of $10$.
For $l>10$, let
\begin{equation}\label{pdefgen}
p_l(j)
=
\left\{
\begin{array}{ll}
\inf\Big\{p>p_c\ :\ L(p)\leq\frac{l}{C_*\log^{(j)}l}\Big\} & \textrm{if } j\in (0, \log^*l),\\
p_c & \textrm{if } j\geq\log^*l,\\
1 & \textrm{if } j=0.
\end{array}\right.
\end{equation}
The value of $C_*$ will be chosen later.
Note that there exists a universal constant $L_0(C_*)>10$ such that $p_l(j)$ are well-defined if $l>L_0(C_*)$ and non-increasing in $l$.  The last observation follows from monotonicity of $L(p)$ and the fact that the functions $l/\log^{(j)}l$ are non-decreasing in $l$ for $j \in (0,\log^*l)$ and $l \geq 3$.

We give the following results without proofs.

\begin{enumerate}
\item
(\cite[(2.10)]{Jarai})
There exists a universal constant $D_1$ such that, for every $l>L_0(C_*)$ and $j\in (0,\log^*l)$,
\begin{eqnarray}\label{corineq}
C_*\log^{(j)}l\leq\frac{l}{L(p_l(j))}\leq D_1C_*\log^{(j)}l.
\end{eqnarray}
\item
(\cite[Theorem~2]{kesten})
There is a constant $D_2$ such that, for all $p>p_c$,
\begin{eqnarray}\label{thetacorrineq}
\theta(p)
\leq
\PP_{p}\big[0\leftrightarrow\partial B(L(p))\big]
\leq
D_2 \PP_{cr}\big[0\leftrightarrow\partial B(L(p))\big],
\end{eqnarray}
where $\theta(p)=\PP_p(0\to\infty)$ is the percolation function for Bernoulli percolation.
\item
(\cite[Section~4]{Nguyen})
There is a constant $D_3$ such that, for all $n\geq 1$,
\begin{equation}\label{eqNguyen}
{\mathbb P}_{p_n}(B(n)\leftrightarrow\infty)\geq D_3.
\end{equation}
\item
(\cite[(3.61)]{kesten})
There is a constant $D_4$ such that, for all positive integers $r\leq s$,
\begin{eqnarray}\label{ineqCrossing}
\frac{\PP_{cr}(0\leftrightarrow \partial B(s))}{\PP_{cr}(0\leftrightarrow \partial B(r))}
\geq
D_4\sqrt{\frac{r}{s}}.
\end{eqnarray}
\item
Recall the definition of $B_n$ from Section~\ref{secNotation}.
There exist positive constants $D_5$ and $D_6$ such that, for all $p> p_c$,
\begin{equation}\label{eqExpDecay}
{\mathbb P}_p(B_{n})
\leq
D_5 \exp\left\{-D_6\frac{n}{L(p)}\right\}.
\end{equation}
It follows, for example, from \cite[(2.6) and (2.8)]{Jarai} (see also \cite[Lemma~37 and Remark~38]{Nolin}).
\item
(\cite[Proposition~34]{Nolin})
Fix $e=\langle(0,0),(1,0)\rangle$, and let $A^{2,2}_n$ be the event that $e_x$ and $e_y$ are connected to $\partial B(n)$ by open paths,
and $e_x^*$ and $e_y^*$ are connected to $\partial B(n)^*$ by closed dual paths. Note that these four paths are disjoint and alternate.
Then
\begin{equation}\label{ineqKesten}
(p_n - p_c)n^2 {\mathbb P}_{cr}(A^{2,2}_n) \asymp 1,~~~~n\geq 1.
\end{equation}
\end{enumerate}

\section{Proof of Theorem~\ref{thmTau}}\label{tausec}
\noindent

We give the proof for the case $k=2$. The proof for $k\geq 3$ is similar to the proof for $k=2$, and we omit the details.
Note that \cite[Theorem 2]{kesten} it is sufficient to prove that
\begin{equation}\label{eqTau1}
{\mathbb P}(\hat\tau_2 < p)\asymp (\log L(p)) {\mathbb P}_{cr}\left(0\leftrightarrow \partial B(L(p))\right).
\end{equation}
We first prove the upper bound. Recall the definition of the radius $\hat R_k$ for $k\geq 1$ from Section~\ref{secNotation}.  We partition the box $B(L(p))$ into $\lfloor \log L(p)\rfloor$ disjoint annuli:
\begin{eqnarray*}
{\mathbb P}(\hat\tau_2<p)
&\leq
&{\mathbb P}(\hat R_1 \geq L(p))\\
&+
&\sum_{k=0}^{\lfloor\log L(p)\rfloor}
{\mathbb P}\left(\hat\tau_2 < p; \hat R_1\in \left[\frac{L(p)}{2^{k+1}},\frac{L(p)}{2^k}\right)\right).
\end{eqnarray*}
We show that there is a universal constant $C_1$ such that for any $p>p_c$ and $m\leq L(p)/2$,
\begin{equation}\label{eqTauUpper1}
{\mathbb P}\left(\hat\tau_2 < p; \hat R_1\in [m,2m]\right)
\leq
C_1 {\mathbb P}_{cr}\left(0\leftrightarrow \partial B(L(p))\right).
\end{equation}
From \cite{pond}, ${\mathbb P}(\hat R_1 \geq L(p))\leq C_2 {\mathbb P}_{cr}\left(0\leftrightarrow \partial B(L(p))\right)$.  Therefore
the upper bound in inequality (\ref{eqTau1}) will immediately follow from~(\ref{eqTauUpper1}).
We partition the event $\{ \hat\tau_2 < p; \hat R_1\in [m,2m] \}$ according to the value of $\hat\tau_1$:
\begin{eqnarray}\label{eqTauUpper2}
\sum_{j=1}^{\log^*m}
{\mathbb P}\left(\hat\tau_2 < p; \hat R_1\in [m,2m]; \hat\tau_1 \in [p_m(j),p_m(j-1))\right).
\end{eqnarray}
Note that if the event $\{\hat R_1 \geq m, \hat\tau_1 \in [p_m(j),p_m(j-1))\}$ occurs then
\begin{itemize}
\item[-]
there is a $p_m(j-1)$- open path from the origin to $\partial B(m)$, and
\item[-]
the origin is surrounded by a $p_m(j)$-closed circuit of diameter at least $m$ in the dual lattice.
\end{itemize}
We also note that if the event $\{\hat\tau_2<p, \hat R_1 \leq 2m\}$ occurs then
there is a $p$-open path from $B(2m)$ to $\partial B(L(p))$.

Recall the definitions of $B_{n,p}$ and $B_n$ from Section~\ref{secNotation}.
From the above observations, it follows that the sum (\ref{eqTauUpper2}) is bounded from above by
$$
\sum_{j=1}^{\log^*m}
{\mathbb P}\left(
0\stackrel{p_m(j-1)}\longleftrightarrow \partial B(m); B(2m)\stackrel{p}\longleftrightarrow\partial B(L(p)); B_{m,p_m(j)}
\right).
$$
The FKG inequality and independence give an upper bound of
$$
\sum_{j=1}^{\log^*m}
{\mathbb P}_{p_m(j-1)}(0\leftrightarrow \partial B(m))
{\mathbb P}_p(B(2m)\leftrightarrow\partial B(L(p)))
{\mathbb P}_{p_m(j)}(B_{m}).
$$
It follows from (\ref{corineq}) and (\ref{eqExpDecay}) that
${\mathbb P}_{p_m(j)}(B_{m}) \leq C_3 (\log^{(j-1)}m)^{-C_4}$ for some $C_3$ and $C_4$, where $C_4$ can be made arbitrarily large
given that $C_*$ is made large enough.  Inequality (\ref{ineqCrossing}) gives
$$
{\mathbb P}_{p_m(j-1)}(0\leftrightarrow \partial B(m))
\leq C_5 (\log^{(j-1)}m)^{\frac{1}{2}} {\mathbb P}_{cr}(0\leftrightarrow\partial B(m)),
$$
and (\ref{thetacorrineq}) and the RSW Theorem give
$$
{\mathbb P}_p(B(2m)\leftrightarrow\partial B(L(p)))
\leq
C_6 {\mathbb P}_{cr}(B(2m)\leftrightarrow\partial B(L(p))).
$$
Also, the RSW Theorem and the FKG inequality imply that
$$
{\mathbb P}_{cr}(0\leftrightarrow\partial B(m))
{\mathbb P}_{cr}(B(2m)\leftrightarrow\partial B(L(p)))
\leq
C_7
{\mathbb P}_{cr}(0\leftrightarrow\partial B(L(p))).
$$
Therefore, we obtain that the probability ${\mathbb P}\left(\hat\tau_2 < p; \hat R_1\in [m,2m]\right)$
is bounded from above by
$$
C_8{\mathbb P}_{cr}(0\leftrightarrow\partial B(L(p)))\sum_{j=1}^{\log^*m}(\log^{(j-1)}m)^{-C_4 + 1/2}.
$$
As in \cite[(2.26)]{Jarai}, one can easily show that, for $C_4>1$,
$$
\sum_{j=1}^{\log^*m}(\log^{(j-1)}m)^{-C_4 + 1/2} < C_9.
$$
The upper bound in (\ref{eqTau1}) follows.

\bigskip
We now prove the lower bound in (\ref{eqTau}). For $p>p_c$ and a positive integer $m<L(p)/2$,
we consider the event $C_{m,p}$ that there exists an edge $e\in Ann(m,2m)$ such that
\begin{itemize}
\item[-]
$\tau_e\in(p_c,p_m)$;
\item[-]
there exist two $p_c$-open paths in $B(2L(p))\setminus \{e\}$,
one connecting the origin to one of the endpoints of $e$,
and another connecting the other endpoint of $e$ to the boundary of $B(2L(p))$;
\item[-]
there exists a $p_m$-closed dual path $P$ in $Ann(m,2m)^*\setminus\{e^*\}$ connecting the endpoints of $e^*$ such that
$P\cup\{e^*\}$ is a circuit around the origin;
\item[-]
there exists a $p_c$-open circuit around the origin in $Ann(L(p),2L(p))$;
\item[-]
there exists a $p$-open path connecting $B(L(p))$ to infinity.
\end{itemize}

\begin{figure}
\begin{center}
\scalebox{.6}{\includegraphics*[viewport = 0in 0in 6.5in 6in]{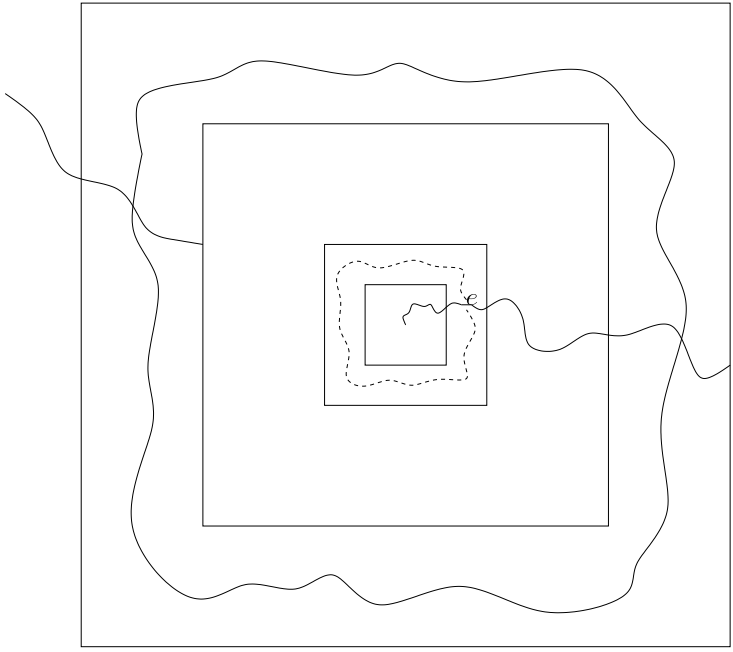}}
\end{center}
\caption{The event $C_{m,p}$.  The boxes, in order from smallest to largest, are $B(m),~B(2m),~B(L(p)),~$and$~B(2L(p))$.  The dotted path is $p_m$-closed, the path to infinity is $p_m$-open, and all other paths are $p_c$-open.}
\label{cmpfig}
\end{figure}
See Figure~\ref{cmpfig} for an illustration of the event $C_{m,p}$.  It can be shown similarly to \cite[Corollary~6.2]{DSV} that
$$
{\mathbb P}(C_{m,p}) \geq C_{10} {\mathbb P}_{cr}(0\leftrightarrow\partial B(L(p))),
$$
where we also use the fact that ${\mathbb P}_p(B(L(p))\leftrightarrow\infty) >C_{11}$ (see (\ref{eqNguyen})).
It remains to notice that for fixed $p$, the events $C_{\lfloor L(p)/2^k\rfloor,p}$ are disjoint and each of them implies the event $\{\hat\tau_2<p\}$.
Therefore,
$$
{\mathbb P}(\hat\tau_2<p)
\geq
\sum_{k=0}^{\lfloor\log L(p)\rfloor-1} {\mathbb P}(C_{\lfloor L(p)/2^k\rfloor,p})
\geq C_{10}\lfloor\log L(p)\rfloor {\mathbb P}_{cr}(0\leftrightarrow\partial B(L(p))).
$$
The lower bound is proven.

\section{Proof of Theorem~\ref{thmVolumes}}\label{volumessec}
\noindent

The case $k = 1$ is considered in \cite[Theorem 2]{pond}. We give the proof for $k = 2$.
The proof for $k\geq 3$ is similar to the proof for $k=2$, and we omit the details.

We first prove the upper bound.
By the RSW Theorem, it is sufficient to bound the probability ${\mathbb P}(|\hat V_2|\geq 2 n^2\pi_n)$.
We partition this probability according to the value of the radii $\hat R_1$ and $\hat R_2$, defined in Section~\ref{secNotation}.
Without loss of generality we can assume that $n = 2^N$.
\begin{eqnarray*}
{\mathbb P}(|\hat V_2|\geq 2 n^2\pi_n)
&\leq
&{\mathbb P}(\hat R_2 \geq n)\\
&+
&\sum_{m=1}^{N}\sum_{k=1}^{m}
{\mathbb P}\left(
|\hat V_2|\geq 2 n^2\pi_n; \hat R_1 \in [2^{k-1},2^{k}); \hat R_2 \in [2^{m-1},2^{m})
\right).
\end{eqnarray*}
It follows from \cite{DSV} that ${\mathbb P}(\hat R_2 \geq n)\asymp (\log n)\pi_n$.
We now consider the second term. We decompose the probability of the event
$$
E_{n,k,m}
=
\left\{
|\hat V_2|\geq 2 n^2\pi_n; \hat R_1 \in [2^{k-1},2^{k}); \hat R_2 \in [2^{m-1},2^{m})
\right\}
$$
according to the values of $\hat\tau_1$ and $\hat\tau_2$:
\begin{equation}\label{eqVolUp1}
\sum_{i=1}^{\log^*2^k}\sum_{j=1}^{\log^*2^m}
{\mathbb P}\left(
E_{n,m,k}; \hat\tau_1\in [p_{2^k}(i),p_{2^k}(i-1)); \hat\tau_2\in [p_{2^m}(j),p_{2^m}(j-1))
\right).
\end{equation}
We consider the event $D_{n,k,m}$ that the number of vetices in the annulus $Ann(2^k,2^m)$ connected
to $B(2^k)$ inside $Ann(2^k,2^m)$ is at least $n^2\pi_n$.
If the vertices in the definition of $D_{n,k,m}$ are connected to $B(2^k)$ by $p$-open paths,
we denote the corresponding event by $D_{n,k,m}(p)$.
We also consider the event $D_{n,k}$ that the number of vertices in the box $B(2^k)$ connected to
the boundary $\partial B(2^k)$ is at least $n^2\pi_n$. If the vertices in the definition of $D_{n,k}$
are connected to $\partial B(2^k)$ by $p$-open paths, we denote the corresponding event by $D_{n,k}(p)$.  Recall the definition of $B_{n,p}$ from Section~\ref{secNotation}.  The probability of a typical summand in (\ref{eqVolUp1}) can be bounded from above by
\begin{equation*}
{\mathbb P}\left(\begin{array}{c}
B_{2^{k-1},p_{2^k}(i)}; B_{2^{m-1},p_{2^m}(j)};
0\stackrel{p_{2^k}(i-1)}\longleftrightarrow \partial B(2^{k});
B(2^k)\stackrel{p_{2^m}(j-1)}\longleftrightarrow\infty; \\[12pt]
D_{n,k,m}(p_{2^m}(j-1)) \cup D_{n,k}(p_{2^k}(i-1))
\end{array}
\right),
\end{equation*}
where we use the fact that $\hat\tau_1 > \hat\tau_2$ a.s.

We use the FKG inequality and independence to estimate the above probability.
It is no greater than
\begin{eqnarray}
&&{\mathbb P}\left(
B_{2^{k-1},p_{2^k}(i)}; B_{2^{m-1},p_{2^m}(j)}
\right)
{\mathbb P}_{p_{2^k}(i-1)}\left(0 \leftrightarrow \partial B(2^k)
\right)
{\mathbb P}_{p_{2^m}(j-1)}\left(
B(2^k)\leftrightarrow\infty;
D_{n,k,m}
\right)\label{eqD1}\\
&&+
{\mathbb P}\left(
B_{2^{k-1},p_{2^k}(i)}; B_{2^{m-1},p_{2^m}(j)}
\right)
{\mathbb P}_{p_{2^k}(i-1)}\left(0 \leftrightarrow \partial B(2^{k}); D_{n,k}
\right)
{\mathbb P}_{p_{2^m}(j-1)}\left(
B(2^k)\leftrightarrow\infty
\right).\label{eqD2}
\end{eqnarray}
The probability ${\mathbb P}(B_{2^{k-1},p_{2^k}(i)}; B_{2^{m-1},p_{2^m}(j)})$ is bounded from above by \cite[(6.6)]{DSV}

\[ C_1(\log^{(i-1)}2^k)^{-C_2}(\log^{(j-1)}2^m)^{-C_2}, \] where the constant $C_2$ can be made arbitrarily large given
$C_*$ is made arbitrarily large.

We first estimate (\ref{eqD1}).
It follows from (\ref{ineqCrossing}) that
\begin{eqnarray*}
{\mathbb P}_{p_{2^k}(i-1)}\left(0 \leftrightarrow \partial B(2^k)\right)
&\leq
&C_3 (\log^{(i-1)}2^k)^{1/2}{\mathbb P}_{cr}(0\leftrightarrow \partial B(2^k))\\
&\leq
&C_4 (\log^{(i-1)}2^k)^{1/2}(\log^{(j-1)}2^m)^{1/2}
\frac{{\mathbb P}_{cr}(0\leftrightarrow \partial B(2^m))}
{{\mathbb P}_{p_{2^m}(j-1)}(B(2^k)\leftrightarrow \infty)}.
\end{eqnarray*}
Substitution gives the following upper bound for (\ref{eqD1}):
\[
{\mathbb P}_{cr}(0\leftrightarrow \partial B(2^m))
C_1C_4(\log^{(i-1)}2^k\log^{(j-1)}2^m)^{-C_2+1/2}
{\mathbb P}_{p_{2^m}(j-1)}(D_{n,k,m}~|~B(2^k)\leftrightarrow \infty).
\]
We now estimate (\ref{eqD2}).
It follows from the FKG inequality and (\ref{ineqCrossing}) that
\[
{\mathbb P}_{p_{2^k}(i-1)}\left(0 \leftrightarrow \partial B(2^k) \right)
\leq
C_5 (\log^{(i-1)}2^k)^{1/2}
{\mathbb P}_{cr}\left(0 \leftrightarrow \partial B(2^k) \right).
\]
Substitution gives the following upper bound for (\ref{eqD2}):
\[
{\mathbb P}_{cr}(0\leftrightarrow \partial B(2^k))
C_{1}C_{5}(\log^{(i-1)}2^k\log^{(j-1)}2^m)^{-C_2+1/2}
{\mathbb P}_{p_{2^k}(i-1)}(D_{n,k} ~|~ 0\leftrightarrow \partial B(2^k)).
\]
Therefore, the sum (\ref{eqVolUp1}) is bounded from above by
\[
C_{6}{\mathbb P}_{cr}(0\leftrightarrow \partial B(2^m))
\sum_{i=1}^{\log^*2^k}\sum_{j=1}^{\log^*2^m}
(\log^{(i-1)}2^k\log^{(j-1)}2^m)^{-C_2+1/2}
{\mathbb P}_{p_{2^m}(j-1)}(D_{n,k,m}~|~B(2^k)\leftrightarrow \infty)
\]
\[
+
C_{6}{\mathbb P}_{cr}(0\leftrightarrow \partial B(2^k))
\sum_{i=1}^{\log^*2^k}\sum_{j=1}^{\log^*2^m}
(\log^{(i-1)}2^k\log^{(j-1)}2^m)^{-C_2+1/2}
{\mathbb P}_{p_{2^k}(i-1)}(D_{n,k} ~|~ 0\leftrightarrow \partial B(2^k)).
\]
Note that \cite[(2.26)]{Jarai} if $C_2 >1/2$, then there exists $C_7>0$ such that for all $k$,
$$
\sum_{i=1}^{\log^*2^k}(\log^{(i-1)}2^k)^{-C_2+1/2}
\leq
C_{7}<\infty.
$$
Also note that analogously to \cite[Lemma 4]{pond} one can show that there exist $C_{8}-C_{11}$ such that,
for all $p>p_c$,
\[
{\mathbb P}_{p}(D_{n,k,m}~|~B(2^k)\leftrightarrow \infty)
\leq
C_{8}\exp\left\{
-C_{9}\frac{n^2\pi_n}{2^{2m}\pi(2^m,p)}
\right\}
\]
and
\[
{\mathbb P}_{p}(D_{n,k} ~|~ 0\leftrightarrow \partial B(2^k))
\leq
C_{10}\exp\left\{
-C_{11}\frac{n^2\pi_n}{2^{2k}\pi(2^k,p)}
\right\},
\]
where $\pi_n$ and $\pi(n,p)$ are defined in Section~\ref{secNotation}.  In particular,
\begin{eqnarray*}
{\mathbb P}_{p_{2^m}(j-1)}(D_{n,k,m}~|~B(2^k)\leftrightarrow \infty)
&\leq
&C_{8}\exp\left\{
-C_{9}\frac{n^2\pi_n}{2^{2m}\pi(2^m,p_{2^m}(j-1))}
\right\}\\
&\leq
&C_{8}\exp\left\{
-C_{12}\frac{n^2\pi_n}{2^{2m}\pi_{2^m}}(\log^{(j-1)}2^m)^{-1/2}
\right\},
\end{eqnarray*}
and, similarly,
\begin{eqnarray*}
{\mathbb P}_{p_{2^k}(i-1)}(D_{n,k} ~|~ 0\leftrightarrow \partial B(2^k))
&\leq
&C_{10}\exp\left\{
-C_{11}\frac{n^2\pi_n}{2^{2k}\pi(2^k,p_{2^k}(i-1))}
\right\}\\
&\leq
&C_{10}\exp\left\{
-C_{13}\frac{n^2\pi_n}{2^{2k}\pi_{2^k}}(\log^{(i-1)}2^k)^{-1/2}
\right\}.
\end{eqnarray*}

Therefore, the sum $\sum_{m=1}^{N}\sum_{k=1}^{m}{\mathbb P}(E_{n,k,m})$ is not bigger than
\begin{equation*}
C_{14}(\log n)\pi_n
\sum_{m=1}^N \frac{\pi_{2^m}}{\pi_n}
\sum_{j=1}^{\log^*2^m}
(\log^{(j-1)}2^m)^{-C_2+1/2}
\exp\left\{
-C_{12}\frac{n^2\pi_n}{2^{2m}\pi_{2^m}}(\log^{(j-1)}2^m)^{-1/2}
\right\}
\end{equation*}
\begin{equation}
\label{tempeq1}
+C_{14}(\log n)\pi_n
\sum_{k=1}^N \frac{\pi_{2^k}}{\pi_n}
\sum_{i=1}^{\log^*2^k}
(\log^{(i-1)}2^k)^{-C_2+1/2}
\exp\left\{
-C_{13}\frac{n^2\pi_n}{2^{2k}\pi_{2^k}}(\log^{(i-1)}2^k)^{-1/2}
\right\},
\end{equation}
where $\log n$ comes from the fact that $\sum_{k=1}^m1 = m \leq N = \log n$.
Finally, it follows from \cite[p.~419]{pond} that
$$
\sum_{m=1}^N \frac{\pi_{2^m}}{\pi_n}
\sum_{j=1}^{\log^*2^m}(\log^{(j-1)}2^m)^{-C_2+1/2}
\exp\left\{-C_{12}\frac{n^2\pi_n}{2^{2m}\pi_{2^m}}(\log^{(j-1)}2^m)^{-1/2}\right\}
\leq
C_{15}<\infty.
$$
A similar bound holds for the summand~(\ref{tempeq1}).  The proof for the second inequality in (\ref{eqVolumes}) is completed.

\bigskip
We now prove the first inequality in (\ref{eqVolumes}).
For $m\leq N$, let $C_{n,m}$ be the event that there exists an edge in $Ann(2^{m-1},2^m)$ such that
\begin{itemize}
\item[-]
its weight $\tau_e\in (p_c,p_{2^m})$;
\item[-]
there exist two disjoint $p_c$-open paths, one connecting an end of $e$ to the origin, and one connecting
the other end of $e$ to $\partial B(2n)$;
\item[-]
there exist a $p_{2^m}$-closed dual path connecting the edges of $e^*$ in $Ann(2^{m-1},2^m)^*$;
\item[-]
there exists a $p_c$-open circuit in $Ann(n,2n)$.
\end{itemize}
It can be shown similarly to \cite[Corollary~6.2]{DSV} that ${\mathbb P}(C_{n,m})\asymp \pi_n$.
We also note that the events $C_{n,m}$ are disjoint
and each of them implies the event $\{\hat R_2\geq n\}$.
Using the arguments from the proof of \cite[Corollary~6.2]{DSV}, it follows that, for any $x\in Ann(2^{N-1},n) =: A_n$ and $1\leq m\leq N-1$,
\[
{\mathbb P}(x\stackrel{p_c}\longleftrightarrow \partial B(2n)~|~C_{n,m})
\geq
C_{16}\pi_n,
\]
from which we conclude that
\begin{equation}\label{eqFirstMom}
{\mathbb E}(|(\hat V_1\cup\hat V_2)\cap A_n| ~|~C_{n,m})
\geq
C_{17}n^2\pi_n.
\end{equation}
We will show later that, for $1\leq m \leq N-2$,
\begin{equation}\label{eqSecMom}
{\mathbb E}(|(\hat V_1\cup\hat V_2)\cap A_n|^2 ~|~C_{n,m})
\leq
C_{18}\left({\mathbb E}(|(\hat V_1\cup\hat V_2)\cap A_n| ~|~C_{n,m})\right)^2.
\end{equation}
If (\ref{eqSecMom}) holds, the second moment estimate gives that, for some $C_{19}>0$,
\[
{\mathbb P}(|\hat V_1\cup\hat V_2|\geq C_{19}n^2\pi_n;C_{n,m})
\geq
C_{19}{\mathbb P}(C_{n,m})
\geq
C_{20}\pi_n.
\]
Therefore
\begin{eqnarray*}
{\mathbb P}(|\hat V_1\cup\hat V_2|\geq C_{19}n^2\pi_n; \hat R_2\geq n)
&\geq
&\sum_{m=1}^{N-2}{\mathbb P}(|\hat V_1\cup\hat V_2|\geq C_{19}n^2\pi_n; C_{n,m})\\
&\geq
&C_{20}(N-2)\pi_n.
\end{eqnarray*}
In particular, we obtain ${\mathbb P}(|\hat V_1\cup\hat V_2|\geq n^2\pi_n) \geq C_{21}(\log n) \pi_n$.
Recall that ${\mathbb P}(|\hat V_1|\geq n^2\pi_n) \asymp \pi_n$.
It immediately gives the inequality
${\mathbb P}(|\hat V_2|\geq n^2\pi_n) \geq C_{22}(\log n) \pi_n$.

It remains to prove (\ref{eqSecMom}). Note that
\begin{equation}\label{eqVol3}
{\mathbb E}(|(\hat V_1\cup\hat V_2)\cap A_n|^2 ~|~C_{n,m})
=
\sum_{x,y\in A_n}{\mathbb P}(x,y\in \hat V_1~|~C_{n,m})
+
\sum_{x,y\in A_n}{\mathbb P}(x,y\in \hat V_2~|~C_{n,m}),
\end{equation}
where we use the fact that, by construction, $\hat V_1$ and $\hat V_2$ cannot both intersect $A_n$.
We estimate the two sums on the r.h.s. separately.
We only consider the first sum.
The other sum is treated similarly.
We decompose the probability ${\mathbb P}(x,y\in \hat V_1 ; C_{n,m})$ according to the value of $\hat\tau_1$:
\[
\sum_{j=1}^{\log^*n}{\mathbb P}(x,y\in \hat V_1; C_{n,m}; \hat\tau_1\in [p_n(j),p_n(j-1))).
\]
Using arguments as in the first part of the proof of this theorem, the above sum is bounded from above by
\begin{eqnarray*}
&&\sum_{j=1}^{\log^*n}{\mathbb P}\left(
\begin{array}{c}
0\stackrel{p_c}\longleftrightarrow\partial B(2^{m-1}); B(2^m)\stackrel{p_c}\longleftrightarrow\partial B(2^{N-2}); \\[5pt]
x\stackrel{p_n(j-1)}\longleftrightarrow \partial B(x,2^{N-2});y\stackrel{p_n(j-1)}\longleftrightarrow\partial B(y,2^{N-2});\\[5pt]
B_{n,p_n(j)}
\end{array}
\right)\\
&\leq&
{\mathbb P}_{cr}(0\leftrightarrow\partial B(2^{m-1})){\mathbb P}_{cr}( B(2^m)\leftrightarrow\partial B(2^{N-2}))\\
&&\sum_{j=1}^{\log^*n}
{\mathbb P}_{p_n(j)}(B_n)
{\mathbb P}_{p_n(j-1)}(x\leftrightarrow \partial B(x,2^{N-2});y\leftrightarrow\partial B(y,2^{N-2})),
\end{eqnarray*}
where $B_{n,p}$ and $B_n$ are defined in Section~\ref{secNotation}.  Again, using tools from the first part of the proof of this theorem (see also the proof of Theorem~1.5 in \cite{DSV}), the above sum is no greater than
\[
C_{23}{\mathbb P}_{cr}(0\leftrightarrow \partial B(n))
{\mathbb P}_{cr}(x\leftrightarrow \partial B(x,2^{N-2});y\leftrightarrow\partial B(y,2^{N-2})).
\]
Similar arguments apply to the second sum in (\ref{eqVol3}).
Since ${\mathbb P}(C_{m,n})\asymp \pi_n$, we get
\[
{\mathbb E}(|(\hat V_1\cup\hat V_2)\cap A_n|^2 ~|~C_{n,m})
\leq
C_{24}\sum_{x,y\in A_n}{\mathbb P}_{cr}\left(x\leftrightarrow \partial B(x,2^{N-2}), y\leftrightarrow \partial B(y,2^{N-2})\right).
\]
The last sum is bounded from above by $C_{25}n^4\pi_n^2$
(see, e.g., the proof of Theorem 8 in \cite{KestenIIC}), which along with (\ref{eqFirstMom}) gives (\ref{eqSecMom}).

\section{Proof of Theorems~\ref{outletmomentthm} - \ref{outletdensitythm}}\label{almostsuresec}

\subsection{Proof of Theorem~\ref{outletmomentthm}}\label{outletmomentsec}

We will use the following lemma.  For $m, n \geq 1$, and $p \in [0,1]$, let $N(m,n,p)$ be the number of edges $e$ in the annulus $Ann(n,2n)$ such that (a) $e$ is connected to $\partial B(e_x,m)$ (where $e_x$ is defined in Section~\ref{secNotation}) by two disjoint $p$-open paths, (b) $e^*$ is connected to $\partial B(e_x,m)^*$ by two disjoint $p_c$-closed paths, (c) the open and closed paths are disjoint and alternate, and (d) $\tau_e \in [p_c,p]$.

\begin{lemma}
Let $m$ be such that $m \leq L(p)$ and $m \leq n$.  There exists $C_1$ such that for all $t,n$,

\begin{equation}
\label{fourarmmoment}
{\E}(N(m,n,p)^t) \leq t!(C_1\frac{n}{m})^{2t}.
\end{equation}

\end{lemma}

\begin{proof}
The proof is very similar to the proof of the upper bound in Theorem 8 in \cite{KestenIIC}, where we need to use \cite[Lemma~6.3]{DSV} to deal with $p$-open paths.
We omit the details.
\end{proof}

To continue the proof of Theorem~\ref{outletmomentthm}, define for $n \geq 1$ and $k$ with $0 \leq k \leq \log^* n$, the event 

\begin{equation}
\label{Hnkdefeq}
H_{n,k} = \left\{
\begin{array}{c}
\textrm{There exists a }p_n(k)- \textrm{open circuit in }Ann(n/4,n/2)\\
\textrm{which is connected to infinity by a }p_n(k)- \textrm{open path.}
\end{array} \right\},
\end{equation}
where $p_n(k)$ is defined in (\ref{pdefgen}).  Let us decompose the $t^{th}$ moment of $O(n,2n)$ according to the events $H_{n,k}$.  By (\ref{corineq}) and (\ref{eqExpDecay}), there exists $C_2, C_3$ such that for all $n,k$,

\begin{equation}
\label{expdecay}
{\P}(H_{n,k}^c) \leq C_2(\log ^{(k-1)}n)^{-C_*C_3}.
\end{equation}
Writing $n_k = \frac{n}{C_*\log^{(k)}n}$ and using the Cauchy-Schwarz inequality for $1<k< \log^* n$,

\[ {\E}(O(n,2n)^t;H_{n,k},H_{n,k+1}^c) \leq ({\E}(N(n_k,n,p_n(k)))^{2t})^{1/2} ({\P}(H_{n,k+1}^c))^{1/2} \]

\[ \leq ((2t)!)^{1/2}(C_1C_*\log ^{(k)}n)^{2t} C_2^{1/2}(\log ^{(k)} n)^{-\frac{C_*C_3}{2}} = (C_2(2t)!)^{1/2}(C_*C_1)^{2t} (\log ^{(k)} n)^{\frac{4t-C_*C_3}{2}}. \]
Choosing $C_* = \frac{4t+2}{C_3}$, this becomes

\[ (C_2(2t)!)^{1/2} \left( \frac{(4t+1)C_1}{C_3} \right) ^{2t} (\log ^{(k)} n)^{-1} \leq (C_4t)^{3t} (\log ^{(k)} n)^{-1} \]
for some $C_4$.  For the case $k=1$, we have

\[ {\E}(O(n,2n)^t;H_{n,1}^c) \leq n^{2t} {\P} (H_{n,1}^c) \leq \frac{C_2}{n} \leq C_2. \]
If we sum over $k$ and bound $\sum_{k=1}^{(\log ^* n) -1} (\log ^{(k)} n)^{-1}$ independent of $n$ as in  \cite[(2.26)]{Jarai}, we get

\[ {\E}(O(n,2n)^t) \leq (Ct)^{3t}. \]
This completes the proof of (\ref{outletmomenteq1}).  To show (\ref{outletmomenteq2}), one must only use Jensen's inequality and (\ref{outletmomenteq1}):

\[ {\E}(O(n,2n)^{t/3}) \leq (Ct)^t. \]
This implies (\ref{outletmomenteq2}). \qed

\subsection{Proof of Theorem~\ref{outletboxthm}}
\label{outletboxsec}
\begin{proof}[Proof of upper bound]
Consider the event $A$ that, for all large $n$, for all $1\leq i\leq n$, the annulus $Ann(2^i,2^{i+c\log n})$ contains
a $p_c$-open circuit around the origin.
Note that ${\mathbb P}(A)=1$ for large enough $c$. We assume that $c$ is an integer. Then $2^{c\log n} = n^c$ is an integer too.

In the annulus $Ann(2^i,2^{i+2c\log n + 1})$, we define the graph ${\mathcal G}^n_i$ as follows.
Let ${\mathcal U}$ be the union of $p_c$-open clusters in $Ann(2^i,2^{i+2c\log n + 1})$ attached to $\partial B(2^{i + 2c\log n +1})$.
In particular, we assume that all the sites in $\partial B(2^{i + 2c\log n +1})$ are in ${\mathcal U}$.
If ${\mathcal U}$ contains a path from $B(2^i)$ to $\partial B(2^{i + 2c\log n +1})$, we define ${\mathcal G}^n_i$ as ${\mathcal U}$.
Otherwise, we consider the invasion percolation cluster ${\mathcal I}$ in $Ann(2^i,2^{i+2c\log n + 1})$ of the invasion percolation process with
$G_0 = B(2^i)$ (that is $B(2^i)$ is assumed to be invaded at step $0$)
terminated at the first time a site from ${\mathcal U}$ is invaded,
and define ${\mathcal G}^n_i$ as ${\mathcal I}\cup{\mathcal U}$.

We say that an edge $e$ is {\it disconnecting} for ${\mathcal G}^n_i$,
if the graph ${\mathcal G}^n_i\setminus\{e\}$ does not contain a path from $B(2^i)$ to $\partial B(2^{i + 2c\log n +1})$.
Let $X^n_i$ be the number of disconnecting edges for ${\mathcal G}^n_i$ in $Ann(2^{i+c\log n},2^{i+c\log n + 1})$.

Note that if the event $A$ occurs then, for all large $n$,
$X^n_i$ dominates $O(2^{i+c\log n},2^{i+c\log n + 1})$, 
the number of outlets of the IPC ${\mathcal S}$ of the origin in $Ann(2^{i+c\log n},2^{i+c\log n + 1})$.

Moreover, for any $i < \lfloor 3c\log n\rfloor$, $(X^n_{i+ k \lfloor 3c\log n\rfloor})_{k=0}^{ \lfloor n/3c\log n\rfloor - 1}$ are independent.

The reader can verify that the proof of Theorem~\ref{outletmomentthm} is
valid when the number of outlets is replaced with $X^n_i$.
Therefore, there exist constants $\lambda>0$ and $C_5 < \infty$ so that, for all $n$ and $i$,
\[{\mathbb E} \exp(\lambda (X^n_i)^{1/3}) < C_5.\]

Let $Y_i$ be a sequence of independent integer-valued random variables with ${\mathbb P}(Y_i>n) = \min\{1,C_5e^{-\lambda n^{1/3}}\}$.
Then, for any $i < \lfloor 3c\log n \rfloor$, $(X^n_{i+ k \lfloor 3c\log n\rfloor})_{k=0}^{\lfloor n/3c\log n\rfloor - 1}$ is stochastically dominated by $(Y_k)_{k=0}^{\lfloor n/3c\log n\rfloor - 1}$.
In particular,
\[
{\mathbb P}(\sum_{i=1}^n X^n_i > C_6 n)
\leq 3c\log n {\mathbb P}(\sum_{i=1}^{\lfloor n/3c\log n\rfloor -1}Y_i > C_6n/3c\log n)
\leq C_7\log n \exp(-C_8 n^{C_9}).
\]
The last inequality follows, for example, from \cite{Nagaev}.

Therefore, with probability one, for all large $n$, $\sum_{i=1}^n X^n_i \leq C_6 n$.

Note that, if the event $A$ occurs, then, for all large $n$,
\[O(2^{c\log n},2^n) \leq \sum_{i=1}^n X^n_i \leq C_6 n.\]

Finally, since the event $A$ occurs with probability one,
\[O(2^n) \leq O(2^{c\log n},2^n) + O(2^{c\log(c \log n)},2^{c\log n}) + |B(2^{c\log(c\log n)})| \leq C_{10} n.\]
This completes the proof of the upper bound in (\ref{outletboxeq}).
\end{proof}

\bigskip
\begin{proof}[Proof of lower bound]

For $i \geq 1$, let $G_i$ be the event that there is no $p_{2^i}$-closed dual circuit
around the origin with radius larger than $2^{i+\log i}$, and let
$G$ be the event that $G_i$ occurs for all but finitely many $i$. It
is easy to see (using inequality (\ref{eqExpDecay})) that ${\mathbb P}(G)=1$.

For $i \geq 1$, let $K_i$ be the event that
\begin{itemize}
\item[-]
there exists a $p_{2^i}$-closed dual circuit ${\mathcal C}$ around the
origin in $Ann(2^i,2^{i+1})^*$;
\item[-]
there exists a $p_c$-open circuit ${\mathcal C}'$ around the origin
in $Ann(2^i,2^{i+1})$;
\item[-]
the circuit ${\mathcal C}'$ is connected to $\partial B(2^{i+\log
i})$ by a $p_{2^i}$-open path.
\end{itemize}
See Figure~\ref{GiKifig} for an illustration of the event $G_i \cap K_i$.  Note that ${\mathcal C}'$ is in $B(2^{i+1})\cap ext({\mathcal C})$.
By RSW theorem and (\ref{eqNguyen}),
\[
{\mathbb P}(K_i) > C_{11} > 0,
\]
for some $C_{11}$ that does not depend on $i$.  

\begin{figure}
\begin{center}
\scalebox{.7}{\includegraphics[viewport = 0in 0in 4.25in 4.25in]{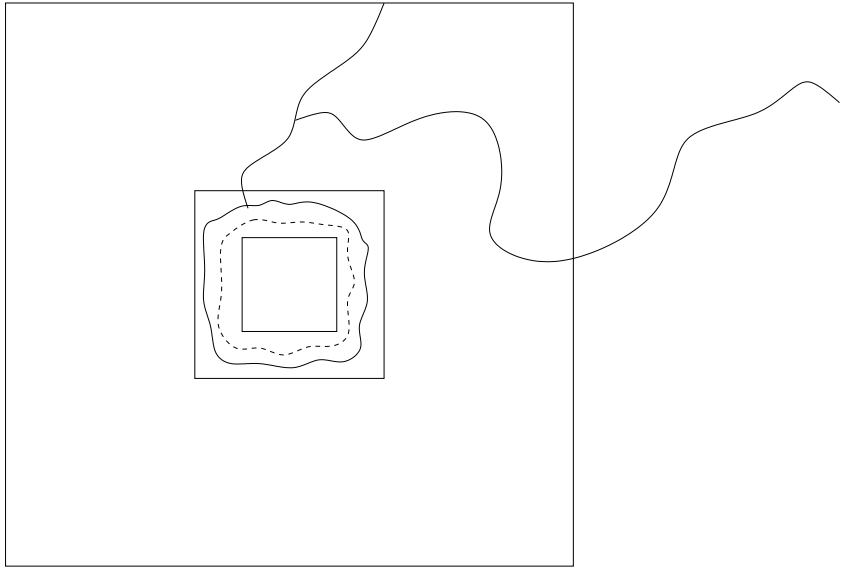}}
\end{center}
\caption{The event $G_i \cap K_i$.  The boxes, in order from smallest to largest, are $B(2^i)$, $B(2^{i+1})$, and $B(2^{i+\log i})$.  Because there is no $p_{2^i}$-closed circuit around the origin of radius larger than $2^{i+\log i}$, the $p_{2^i}$-open path which connects $\partial B(2^{i+\log i})$ to the circuit in $Ann(2^i, 2^{i+1})$ must be connected to $\infty$ by a $p_{2^i}$-open path.}
\label{GiKifig}
\end{figure}

Fix an integer $n$, and let $j$ be an integer between $1$ and $\log
n$. We consider events $K^j_i = K_{j + i\log n}$. Note that, for any
fixed $j$, the events $(K^j_i)_{i=0}^{(\lfloor n/\log n \rfloor)-1}$ are
independent.

Let $X^j_i = I_{K^j_i}$. Recall that ${\mathbb P}(X^j_i = 1) > C_{11}$.
We need the following lemma:
\begin{lemma}
\label{indvarslemma}
Let $c>0$. There exist $\alpha>0$ and $\beta<1$ depending on $c$
with the following property. If $X_i$ are independent $0/1$ random
variables (not necessarily identically distributed) with ${\mathbb
P}(X_i = 1)
> c$ for all $i$, then for all $n$,
\[
{\mathbb P}\left(\sum_{i=1}^n X_i < \alpha n\right) < \beta^n.
\]
\end{lemma}
We first show how to deduce the lower bound of (\ref{outletboxeq}) from this lemma. It
follows that there exist $\alpha > 0$ and $\beta < 1$ such that for
any $n$ and $1\leq j\leq \log n$
\[
{\mathbb P}\left(\sum_{i=0}^{\lfloor n/\log n \rfloor - 1} X^j_i < \frac{\alpha
n}{\log n}\right) < \beta^{n/\log n}.
\]
Therefore,
\[
{\mathbb P}\left(\sum_{j=1}^{\log n} \sum_{i=0}^{\lfloor n/\log n \rfloor - 1} X^j_i
< \alpha n\right) \leq {\mathbb P}\left(\sum_{i=0}^{\lfloor n/\log n \rfloor - 1}
X^j_i < \alpha n/\log n~~\mbox{for some}~j\in[1,\log n]\right) \leq
\log n \beta^{n/\log n}.
\]
In particular, it follows from Borel-Cantelli's lemma that, with
probability one, for all large $n$,
\[
\sum_{i=1}^n I_{K_i} \geq \alpha n.
\]
Finally, observe that the event $G$ occurs with probability one, and
the event $G_i\cap K_i$ implies that there exists an outlet in
$Ann(2^i,2^{i+1})$.  The lower bound in (\ref{outletboxeq}) follows.
\end{proof}

\begin{proof}[Proof Lemma~\ref{indvarslemma}]
Chernov's inequality and the independence of the variables in the set $\{ X_i~:~i \geq 1\}$ give
\[
{\mathbb P}\left(\sum_{i=1}^n X_i < \alpha n\right)\leq
e^{\lambda\alpha n} \prod_{i=1}^n{\mathbb E}e^{-\lambda X_i}.
\]
It is easy to see that there exists $\gamma<1 $, independent of $i$,
such that
\[
{\mathbb E}e^{-\lambda X_i} < \gamma,
\]
for large enough $\lambda$. Now pick $\alpha$ such that
$e^{\lambda\alpha}\gamma<1$.
\end{proof}

\bigskip
\begin{proof}[Proof of Corollary~\ref{outletsascor}]
The inequalities (\ref{outletsaseq1}) follow immediately from those in Theorem~\ref{outletboxthm}.  Therefore we will only prove (\ref{outletsaseq2}).  First we show the upper bound.

Choose $c_5$ from (\ref{outletsaseq1}).  Using (\ref{corineq}) and (\ref{eqExpDecay}), we can show that with probability one, for all large $n$, after the invasion has reached $\partial B(n)$, the weight of each further accepted edge is no larger than $p_n(1)$, where $p_n(1)$ is defined in (\ref{pdefgen}).  Therefore, for all large $k$,

\[ \hat{\tau}_k - p_c \leq p_{(c_5)^k}(1) - p_c. \]
Since there exists $C_{12}, C_{13} >0$ such that for all $n$, $p_n(1) - p_c \leq C_{12}n^{-C_{13}}$ (use (\ref{ineqKesten}) and the fact that the 4-arm exponent is strictly smaller than 2 (see, e.g., Section~6.4 in \cite{Werner})), we have

\[ \hat{\tau}_k - p_c \leq C_{12} (c_5)^{-C_{13}k}, \]
proving the upper bound.  To show the lower bound, choose $c_6$ from (\ref{outletsaseq1}).  For $a<1$, we obtain
\begin{eqnarray*}
{\mathbb P}(\hat\tau_k < p_c + a^k, \hat R_k < (c_6)^k) &\leq &{\mathbb
P}(B((c_6)^k)\stackrel{p_c + a^k}\longleftrightarrow \infty) \leq C_{14}
{\mathbb P}_{cr}(B((c_6)^k)\leftrightarrow \partial B(L(p_c+a^k)))\\ &\leq
&C_{15}\left(\frac{(c_6)^k}{a^{-C_{16}k}}\right)^{C_{17}} \leq C_{18} e^{-C_{19}k},
\end{eqnarray*}
for constants $C_{14}-C_{19}$, where the last inequality holds for small enough $a$. The second
inequality follows from (\ref{thetacorrineq}). The third one follows from, for example, \cite[eq. 11.90]{Grimmett} and the fact that $L(p) > (p - p_c)^{-\delta}$ for some $\delta>0$ (see,
e.g., Cor. 1 and eq. 2.3 from \cite{kesten}). Borel-Cantelli's lemma gives the lower bound of (\ref{outletsaseq2}).

\end{proof}

\subsection{Proof of Theorem~\ref{outletdensitythm}}\label{densitysec}
\noindent
Given any nonempty subinterval of $(0,1]$, we will show that with probability one, $\left( \frac{\hat\tau_{k+1}-p_c}{\hat\tau_k-p_c} \right)$ is in this subinterval for infinitely many $k$.  We will use the following fact.  From \cite[(4.35)]{kesten} it follows that, for any $a>0$,
$$
L(p_c + a\delta)\asymp L(p_c + \delta).
$$
The constants above depend on $a$ but do not depend on $\delta$ so long as $\delta$ is sufficiently small.

Pick a nonempty interval $[x,y] \subset (0,1]$ and choose $a,b>0$ such that

\[ a < b~~\mbox{and}~~bx<ay. \]

We consider the event $D_n$ that there exist $p_c$-open circuits around the origin in the annulus $Ann(n,2n)$; and in the annulus $A(4n,8n)$, and there exist two edges, $e_1 \in Ann(2n,3n)$ and $e_2 \in Ann(8n,9n)$ such that
\begin{itemize}
\item[-]
there is a $p_c$-open path connecting one of the ends of $e_1$ to $B(n)$, and
there is a $p_c$-open path connecting the other end of $e_1$ to $\partial B(8n)$;
\item[-]
there is a $p_c$-open path connecting one of the ends of $e_2$ to $B(4n)$, and
there is a $p_n$-open path connecting the other end of $e_2$ to infinity;
\item[-]
there is a $(p_n + b(p_n-p_c))$-closed path $P_1$ in the dual lattice inside $Ann(2n,3n)^*$ connecting the ends of $e_1^*$ such that
$P_1\cup\{e_1^*\}$ is a circuit around the origin;
\item[-]
there is a $(p_n+ ay(p_n-p_c))$-closed path $P_2$ in the dual lattice inside $Ann(8n,9n)^*$ connecting the ends of $e_2^*$ such that
$P_2\cup\{e_2^*\}$ is a circuit around the origin;
\item[-]
the weight $\tau_{e_1}\in (p_n + a(p_n-p_c), p_n + b(p_n-p_c))$, and the weight $\tau_{e_2} \in (p_n + bx(p_n-p_c), p_n + ay(p_n-p_c))$.
\end{itemize}

\begin{figure}
\begin{center}
\scalebox{.8}{\includegraphics[viewport = 0in 0in 4in 4in]{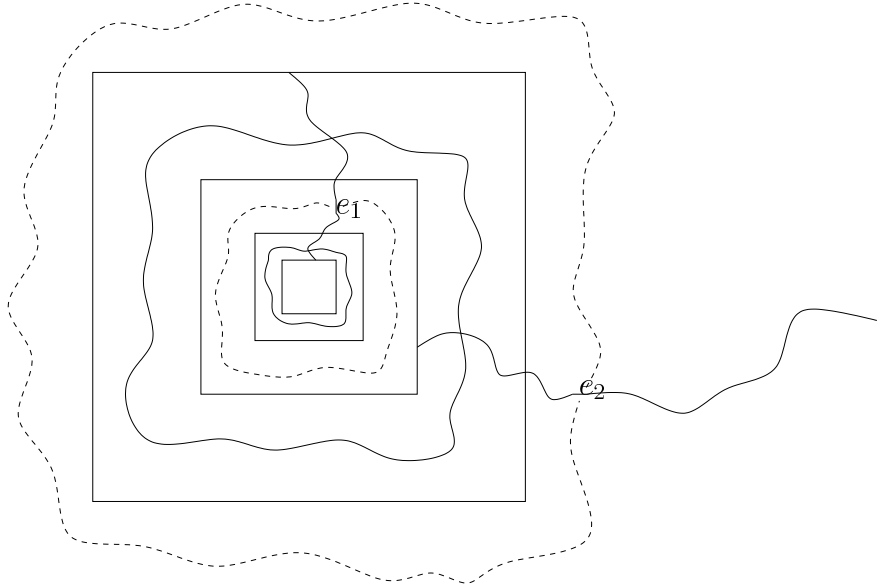}}
\end{center}

\caption{The event $D_n$.  The boxes, from smallest to largest, are $B(n), B(2n), B(4n)$, and $B(8n)$.  The edge $e_1$ is connected to both $B(n)$ and $\partial B(4n)$.  The edge $e_2$ is connected to both $B(4n)$ and infinity.  The solid curves represent occupied paths and the dotted curves represent vacant dual paths.  In the figure, both $e_1$ and $e_2$ are outlets of the invasion.}
\label{Dnfig}
\end{figure}

See Figure~\ref{Dnfig} for an illustration of the event $D_n$.  By RSW arguments and \cite[Lemma 6.3]{DSV} (similar to the proof of \cite[Corollary~6.2]{DSV}), there exists a constant $C_{20} >0$ which depends on $a,b,x,$ and $y$ but not on $n$ such that

\begin{equation}
\label{densityeq1}
{\mathbb P}(D_n) \geq C_{20}.
\end{equation}

We consider the event $\limsup_n D_n$ that there are infinitely many $n$
for which $D_n$ occurs. Since $\limsup_n D_n$ does not depend on the
states of finitely many edges, ${\mathbb P}(\limsup_n D_n) \in \{0,1\}$.
Assume that this probability is $0$.  Then there exists $N$ (deterministic) such that

\[ {\P}(D_n \textrm{ occurs for some } n \geq N) < C_{20}/2. \]
But this probability is, in fact, at least ${\P}(D_N)$.  This contradicts (\ref{densityeq1}).  Therefore

\begin{equation}
\label{boreleq1}
{\mathbb P}(\limsup_nD_n) = 1.
\end{equation}

Note that the event $D_n$ implies that there exists $k = k(n)$ such that $e_1$ and $e_2$ are respectively the $k^{th}$ and $(k+1)^{st}$ outlets of the invasion.  In particular, using the above bounds for $\tau_{e_1}$ and $\tau_{e_2}$,

$$
\frac{\hat\tau_{k+1}-p_c}{\hat\tau_{k}-p_c} \in [x,y].
$$
Combining this with (\ref{boreleq1}), we get

$$
{\P} \left( \frac{\hat\tau_{k+1}-p_c}{\hat\tau_k - p_c} \in [x,y] \textrm{ for infinitely many } k \right) = 1.
$$
This completes the proof.

\section{Proof of Theorem~\ref{thmIIC}}\label{IICsec}
\noindent

Since the proof is very similar to the proof of Theorem (3) in \cite{KestenIIC}, we only sketch the main ideas.
From now on we fix $\sigma\in\{\mbox{open},\mbox{closed}\}^{|\sigma|}$, and assume that
$\sigma$ consists of $m$ `open' and $m$ `closed'. 

The RSW theorem implies that there exists $\delta>0$ such that for all $N$,
\[
{\mathbb P}_{cr}(\mbox{there exists an occupied circuit in}~Ann(N,2N)) \geq \delta.
\]
Since events depending on the state of edges in disjoint annuli are independent, we can find an increasing sequence $N_i$ such that
\[
\alpha_i
=
{\mathbb P}_{cr}(\mbox{there exists an occupied circuit in}~Ann(N_i,N_{i+1}))
\to 1,
\]
as $i\to \infty$.
We fix the sequence $N_i$ and write $A_i$ for $Ann(N_i,N_{i+1})$.

Let ${\mathcal C}$ be a (self-avoiding) circuit in $({\mathbb Z}^2,{\mathbb E}^2)$.
We say that ${\mathcal C}$ is \textit{occupied with $m$ defects} if all but $m$ edges of ${\mathcal C}$ are occupied. Let $F_i({\mathcal C}_{e_1,\ldots,e_m})$ be the event that there exists an occupied circuit with $m$ defects in $A_i$ around $B(N_i)$, and
moreover, the innermost such circuit is ${\mathcal C}$ with defected edges $e_1,\ldots,e_m$.
We also write
\[
F_i^{(m)} = \cup_{\mathcal C\subset A_i}\cup_{e_1,\ldots,e_m\in{\mathcal C}\cap{\mathbb E}^2}F_i({\mathcal C}_{e_1,\ldots,e_m}).
\]
Note that
\[
{\mathbb P}_{cr}(F_i^{(m)}) = \sum_{\mathcal C\subset A_i}\sum_{e_1,\ldots,e_m\in{\mathcal C}\cap{\mathbb E}^2}{\mathbb P}_{cr}(F_i({\mathcal C}_{e_1,\ldots,e_m})).
\]

Recall from Section~\ref{secIICs} that the number $l$ is defined so that $|\partial B(l)| > |\sigma|$.  Let $E$ be any event depending only on the state of edges in $B(r)$ (where we assume that $r>l$) and let $i$ be such that $r<N_i<N_{i+1}<n$.   
Then
\begin{eqnarray*}
{\mathbb P}_{cr}(E\cap\{B(l)\leftrightarrow_\sigma \partial B(n)\})
&=
&{\mathbb P}_{cr}(E\cap\{B(l)\leftrightarrow_\sigma \partial B(n)\}\cap (F^{(m)}_i)^c)\\
&+
&\sum_{\mathcal C\subset A_i}\sum_{e_1,\ldots,e_m\in{\mathcal C}\cap{\mathbb E}^2}
{\mathbb P}_{cr}(E\cap\{B(l)\leftrightarrow_\sigma \partial B(n)\}\cap F_i({\mathcal C}_{e_1,\ldots,e_m})).
\end{eqnarray*}
Let $\{B(l) \leftrightarrow_\sigma {\cal C}_{e_1,\ldots,e_m}\}$ denote the event that $B(l)$ is $\sigma$-connected to ${\cal C}$ so that the $m$ disjoint closed paths connect $B(l)$ to the edges $e_1^*,\ldots,e_m^*$ in $\stackrel{\circ}{{\cal C}}$ (interior of ${\cal C}$). Similarly, let $\{{\cal C}_{e_1,\ldots,e_m} \leftrightarrow_\sigma \partial B(n)\}$ denote the event that ${\cal C}$ is $\sigma$-connected to $\partial B(n)$ so that the $m$ disjoint closed paths connect $\partial B(n)$ to the edges $e_1^*,\ldots,e_m^*$ in $\bar{{\cal C}}^e$ (exterior of ${\cal C}$).

We now estimate the difference ${\mathbb P}_{cr}(E\cap\{B(l)\leftrightarrow_\sigma \partial B(n)\}\cap (F^{(m)}_i)^c)$
between
\[
{\mathbb P}_{cr}(E\cap\{B(l)\leftrightarrow_\sigma \partial B(n)\})
\]
and
\begin{equation*}
\sum_{\mathcal C\subset A_i}\sum_{e_1,\ldots,e_m\in{\mathcal C}\cap{\mathbb E}^2}
{\mathbb P}_{cr}(E\cap F_i({\mathcal C}_{e_1,\ldots,e_m})\cap \{B(l) \leftrightarrow_\sigma {\cal C}_{e_1,\ldots,e_m}\})
{\mathbb P}_{cr}({\mathcal C}_{e_1,\ldots,e_m} \leftrightarrow_\sigma \partial B(n)).
\end{equation*}
By Menger's theorem \cite[Theorem~3.3.1]{Diestel}, the event $\{B(l)\leftrightarrow_\sigma \partial B(n)\}\cap (F^{(m)}_i)^c$ implies that
there exist $(m+1)$ disjoint closed crossings of the annulus $A_i$.
We use Reimer's inequality \cite{Reimer} to conclude that
the probability
${\mathbb P}_{cr}(E\cap\{B(l)\leftrightarrow_\sigma \partial B(n)\}\cap (F^{(m)}_i)^c)$
is bounded from above by
\begin{eqnarray*}
&&{\mathbb P}_{cr}(B(l)\leftrightarrow_\sigma \partial B(n))
{\mathbb P}_{cr}(\mbox{there exists a closed crossing of}~A_i)\\
&&\leq
(1-\alpha_i){\mathbb P}_{cr}(B(l)\leftrightarrow_\sigma \partial B(n)).
\end{eqnarray*}

We have just shown how a statement similar to (17) in \cite{KestenIIC} is obtained.
An analogous statement to (18) in \cite{KestenIIC} is also valid.
The remainder of the proof is similar to the proof of Kesten \cite{KestenIIC},
where in the proof of the statement analogous to Lemma~(23) in \cite{KestenIIC} we use extensions of arm separation techniques from \cite[Section~4]{Nolin}.

We use the following analogue of Kesten's Lemma (23).
\begin{lemma}\label{l23}
Consider circuits $\mathcal C$ in annulus $A_i$, $\mathcal D$ in annulus $A_{i+3}$, sets of edges $e_1,\ldots,e_m$ on $\mathcal C$ and
$f_1,\ldots,f_m$ on $\mathcal D$ respectively.
Let $P(\mathcal C,\mathcal D)$ be the probability,
conditional on the event that all edges in $\mathcal C\setminus\{e_1,\ldots,e_m\}$ are open and $e_1,\ldots,e_m$ are closed, that
(1) there are disjoint closed dual paths from $e_i^*$ to $f_i^*$,
(2) there are $m$ disjoint open paths that connect $\mathcal C$ to $\mathcal D$ such that, for any two of them, there is a closed dual path (one of the paths from (1))
between them,
(3) $\mathcal D$ is the innermost open circuit with defects $f_1,\ldots,f_m$ in annulus $A_{i+3}$,
(4) there is an open circuit with $m$ defects in annulus $A_{i+2}$.
We similarly define $\mathcal C'$, $\mathcal D'$, etc.
There exists a finite constant $C_1$ that may depend only on $m$
(it does not depend on particular choice of circuits or defects) such that
\[
\frac{P(\mathcal C,\mathcal D)P(\mathcal C',\mathcal D')}{P(\mathcal C,\mathcal D')P(\mathcal C',\mathcal D)}<C_1.
\]
\end{lemma}
To prove Lemma~\ref{l23}, we need the following extension of Kesten's arm separation Lemmas~4 and 5 \cite{kesten}.
Let $\mathcal I$ be a fixed partition of $\partial B(1)$ (in $\mathbb R^2$) into $2m$ disjoint connected subsets $\mathcal I_i$, each of diameter at least $1/(2m)$ (ordered clockwise).
Let $\mathcal I(s)$ be the corresponding partition of $\partial B(s)$ into $2m$ disjoint connected subsets
$\mathcal I_i(s) = s \mathcal I_i = \{s x~:~x\in{\mathcal I}_i\}$.
Let $\mathcal I(n,n')$ be the partition of $\overline{Ann(n,n')}$ into $2m$ disjoint connected subsets
$\mathcal I_i(n,n') = \cup_{n\leq s\leq n'} \mathcal I_i(s)$.
\begin{lemma}[external arm separation]\label{lKesten}
Let $n_0$ and $n$ be positive integers with $n_0 \leq n-3$.
We consider a circuit $\mathcal C$ in $B(2^{n_0})$ and a set of edges $e_1,\ldots,e_m$ on $\mathcal C$.
Let $E(\mathcal C_{e_1,\ldots,e_m})$ be the event that
(1) the edges in $\mathcal C\setminus\{e_1,\ldots,e_m\}$ are open and $e_1,\ldots,e_m$ are closed,
(2) there are $m$ disjoint closed dual paths from $e_j^*$ to $\partial B(2^n)^*$,
(3) there are $m$ disjoint open paths from $\mathcal C$ to the boundary of $B(2^n)$ in $\left( B(2^n)\setminus \mbox{int}(C)\right)\setminus\{e_1,\ldots,e_m\}$
such that these paths alternate with the closed paths defined in (2).
Let $\widetilde E(\mathcal C_{e_1,\ldots,e_m})$ be the event that
$E(\mathcal C_{e_1,\ldots,e_m})$ occurs with $2m$ paths $P_1,\ldots, P_{2m}$ (ordered clockwise, all paths with odd indices are closed, and the ones with even indices are open) satisfying the requirement that, for all $1\leq i\leq 2m$, $P_i \cap \overline{Ann(2^{n-1},2^n)} \subset \mathcal I_i(2^{n-1},2^n)$.
Then
\[
{\mathbb P}(E(\mathcal C_{e_1,\ldots,e_m}))
\leq C_2
{\mathbb P}(\widetilde E(\mathcal C_{e_1,\ldots,e_m})),
\]
where the constant $C_2$ may depend on $m$ but not on $n$, $n_0$, or the choice of circuit.
\end{lemma}
\begin{remark}
The event $\widetilde E$ is reminiscent of the event $\Delta$ in \cite{kesten} (page~127 and Figure~8).
\end{remark}
\begin{remark}
It is actually believed \cite{gabor} and is the aim of ongoing work of Garban and Pete that a much
stronger statement holds: given any configuration inside $B(2^{n_0})$, if we
condition on the existence of $m$ open paths and $m$ closed dual paths from a neighborhood of the origin to $\partial B(2^n)$ and these paths are alternating, then they will be well-separated (refer to \cite{kesten} for this definition) on $\partial B(2^n)$ with positive probability independent of $n$ and the configuration inside $B(2^{n_0})$.
\end{remark}
\begin{lemma}[internal arm separation]\label{lKesten1}
Let $n$ and $n_1$ be positive integers with $n +3 \leq n_1$.
Consider a circuit $\mathcal D$ in $B(2^{n_1})^c$ and a set of edges $f_1,\ldots,f_m$ on $\mathcal D$.
Let $F(\mathcal D_{f_1,\ldots,f_m})$ be the event that
(1) the edges in $\mathcal D\setminus\{f_1,\ldots,f_m\}$ are open and $f_1,\ldots,f_m$ are closed,
(2) there are $m$ disjoint closed dual paths from $f_j^*$ to $B(2^n)^*$,
(3) there are $m$ disjoint open paths from $\mathcal D$ to $B(2^n)$ in $\overline {\mbox{int}(D)}$
such that these paths alternate with the closed paths defined in (2).
Let $\widetilde F(\mathcal D_{f_1,\ldots,f_m})$ be the event that
the event $F(\mathcal D_{f_1,\ldots,f_m})$ occurs with $2m$ paths $P_1,\ldots, P_{2m}$ (ordered clockwise, all paths with odd indices are closed, and the ones with even indices are open) satisfying the requirement that, for all $1\leq i\leq 2m$, $P_i \cap \overline{Ann(2^n,2^{n+1})} \subset \mathcal I_i(2^n,2^{n+1})$.
Then
\[
{\mathbb P}(F(\mathcal D_{f_1,\ldots,f_m}))
\leq C_3
{\mathbb P}(\widetilde F(\mathcal D_{f_1,\ldots,f_m})),
\]
where the constant $C_3$ may depend on $m$ but not on $n$, $n_1$, or the choice of circuit.
\end{lemma}

The proofs of Lemmas~\ref{lKesten} and \ref{lKesten1} are similar, and we only give the proof of Lemma~\ref{lKesten} here.
Moreover, parts of the proof of Lemma~\ref{lKesten} are similar to the proof of Lemma~4 in \cite{kesten}.
We will refer the reader to \cite{kesten} for the proof of those parts.
Before we give the proof of Lemma~\ref{lKesten}, we show how to deduce Lemma~\ref{l23} from the above two lemmas.

Using Lemma~\ref{lKesten}, Lemma~\ref{lKesten1} and ``gluing'' arguments (see \cite{kesten,Nolin}), we prove
\begin{lemma}
For two circuits, ${\mathcal C}_1$ in annulus $A_i$ and
${\mathcal D}_1$ in annulus $A_{i+2}$, sets of edges $e_1,\ldots,e_m$ on ${\mathcal C}_1$ and $f_1,\ldots,f_m$
on ${\mathcal D}_1$, if $M({\mathcal C}_1,{\mathcal D}_1)$ is the probability,
conditioned on the event that all edges in $\mathcal C_1\setminus\{e_1,\ldots,e_m\}$ and in $\mathcal D_1\setminus\{f_1,\ldots,f_m\}$ are open
and $e_1,\ldots,e_m$, $f_1,\ldots,f_m$ are closed, that there
are disjoint closed dual paths from $e_i^*$ to $f_i^*$ for all $i$, and there are
$m$ disjoint open paths from $\mathcal C_1$ to $\mathcal D_1$
in $\overline{\mbox{int}({\mathcal D}_1)}\setminus \mbox{int}({\mathcal C}_1)$, which alternate with the closed paths defined above
(and similar definitions for ${\mathcal C}_2$ and ${\mathcal D}_2$), then
\[
\frac{M({\mathcal C}_1,{\mathcal D}_1)M({\mathcal C}_2,{\mathcal D}_2)}{M({\mathcal C}_1,{\mathcal D}_2)M({\mathcal C}_2,{\mathcal D}_1)}<C_4,
\]
for some constant $C_4$ that does not depend on the particular choice of circuits or defects.
\end{lemma}
\begin{proof}
This lemma follows from Lemma~\ref{lKesten}, Lemma~\ref{lKesten1}, the RSW theorem (Section~11.7 in \cite{Grimmett}), and the generalized FKG inequality (Lemma~3, \cite{kesten}).
For more details we refer the reader to the proof of (2.43) in \cite{kesten}.
\end{proof}
Lemma~\ref{l23} immediately follows from the above lemma.
\begin{proof}[Proof of Lemma~\ref{l23}]
Consider circuits $\mathcal C_1$ in annulus $A_{i+2}$, $\mathcal D_1$ in $A_{i+3}$, sets of edges $g_1,\ldots,g_m$ on $\mathcal C_1$ and $h_1,\ldots,h_m$ on $\mathcal D_1$ respectively. Let $H(\mathcal C_1,\mathcal D_1)$ be the probability of the event that (1) $\mathcal C_1$ is the outermost open circuit with defects $g_1,\ldots,g_m$ in annulus $A_{i+2}$, (2) $\mathcal D_1$ is the innermost open circuit with defects $h_1,\ldots,h_m$ in annulus $A_{i+3}$, (3) there are disjoint closed dual paths from $g_i^*$ to $h_i^*$, and (4) there are $m$ disjoint open paths from $\mathcal C_1$ to $\mathcal D_1$ in $\overline{int(\mathcal D_1)}\setminus int(\mathcal C_1)$,
which alternate with the closed paths defined above.

We write
\[
P({\mathcal C},{\mathcal D})P({\mathcal C}',{\mathcal D}')
=
\sum_{\mathcal C_1}M(\mathcal C,\mathcal C_1) H(\mathcal C_1,\mathcal D) \sum_{\mathcal C_1'}M(\mathcal C',\mathcal C_1')H(\mathcal C_1',\mathcal D').
\]
We then apply the previous lemma to $\mathcal C$, $\mathcal C_1$, $\mathcal C'$ and $\mathcal C_1'$:
\[
P({\mathcal C},{\mathcal D})P({\mathcal C}',{\mathcal D}')
\leq
C_4 \sum_{\mathcal C_1}M(\mathcal C',\mathcal C_1) H(\mathcal C_1,\mathcal D) \sum_{\mathcal C_1'}M(\mathcal C,\mathcal C_1')H(\mathcal C_1',\mathcal D').
\]
\end{proof}

\begin{proof}[Proof of Lemma~\ref{lKesten}]
We only consider the case $m=2$. The case $m=1$ is simpler, and the general case is similar to the case $m=2$.

Let $\mathcal C$ be a circuit in $B(2^{n_0})$. All edges in $\mathcal C$ are open except for two edges $e$ and $f$, which are closed.
Let $P_1$ and $P_3$ be disjoint closed dual paths from $e^*$ and $f^*$ to $\partial B(2^{n})^*$, respectively, and
let $P_2$ and $P_4$ be disjoint open paths from $\mathcal C$ to $\partial B(2^{n})$ that satisfy conditions of the lemma.

We define $\gamma_1^l$ as the leftmost closed dual path from $e^*$ to $\partial B(2^{n_0}+1/2)$ in $B(2^{n_0}+1/2)\setminus \mbox{int}(\mathcal C)$, and
$\gamma_1^r$ as the rightmost closed dual path from $e^*$ to $\partial B(2^{n_0}+1/2)$ in $B(2^{n_0}+1/2)\setminus \mbox{int}(\mathcal C)$.
We denote the first vertex on $\partial B(2^{n_0})$ to the left of $\gamma_1^l$ as $a_1$, and the first vertex on $\partial B(2^{n_0})$ to the right of $\gamma_1^r$ as $a_2$.
Let $\gamma_2^l$ be the leftmost open path from the right end-vertex of $e$
(using the clockwise ordering of vertices end edges on $\mathcal C$) to $a_2$. This path is necessarily contained in $B(2^{n_0})\setminus \mbox{int}(\mathcal C)$. Let $\gamma_4^r$ be the rightmost open path from the left end-vertex of $e$ (using the clockwise ordering of vertices end edges on $\mathcal C$) to $a_1$ in $B(2^{n_0})\setminus \mbox{int}(\mathcal C)$.
Similarly we define $\gamma_3^l$, $\gamma_3^r$, $a_3$, $a_4$, $\gamma_4^l$ and $\gamma_2^r$ (see Figure~\ref{fEventGammas}).

For $i\in \{1,2,3,4\}$, let $T_i$ be the piece of $\partial B(2^{n_0})$ between (and including) $a_i$ and $a_{i+1}$ that does not contain $a_{i+2}$ or $a_{i+3}$, where
we use the convention $a_i = a_{i-4}$ for $i>4$. Note that it is possible that $a_2 = a_3$ (in which case $T_2 = \{a_2\}$) or $a_4 = a_1$ (in which case $T_4 = \{a_4\}$);
however, we necessarily have $a_1 \neq a_2$ and $a_3 \neq a_4$. 

Let $\gamma_i$ be the part of $\gamma_i^l\cup\gamma_i^r\cup\mathcal C$ that consists of the piece of $\gamma_i^l$ from the last intersection with $\gamma_i^r\cup\mathcal C$, the piece of $\gamma_i^r$ from the last intersection with $\gamma_i^l\cup\mathcal C$, and the piece of $\mathcal C$ that connects the first two pieces (if the pieces are disconnected). Note that it is possible that $\gamma_2$ or $\gamma_4$ is a single point set on $\partial B(2^{n_0})$, which happens if $a_2 = a_3$ or $a_4 = a_1$, respectively. Let $R_i$ denote the connected subset of $\R^2$ with the boundary that consists of $T_i$ and $\gamma_i$ (see Figure~\ref{fEventGammas}). Note that these sets are disjoint. Moreover, if $\gamma_2$ or $\gamma_4$ is a single point set ($\{a_2\}$ or $\{a_4\}$ respectively), then $R_2$ or $R_4$ is the same single point set.

\begin{figure}[h]
\begin{center}
\includegraphics[width=12cm]{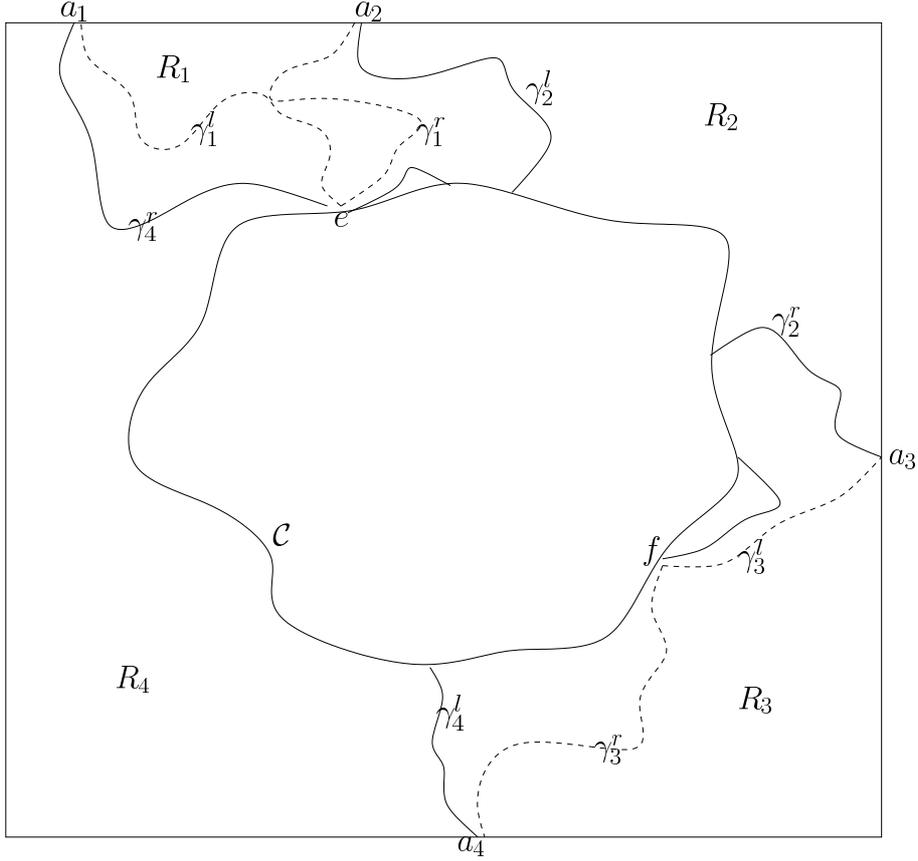}
\caption{The event $E(\mathcal C_{e_1,\ldots,e_m})$ occurs if and only if $\gamma_1$ and $\gamma_3$ are connected by closed paths to $\partial B(2^{n})^*$,
and $\gamma_2$ and $\gamma_4$ are connected by open paths to $\partial B(2^{n})$ in $B(2^{n_0})^c\cup R_1\cup R_2\cup R_3\cup R_4$.}
\label{fEventGammas}
\end{center}
\end{figure}

We observe that under the assumptions of Lemma~\ref{lKesten}, the event $E(\mathcal C_{e_1,\ldots,e_m})$ occurs if and only if $\gamma_1$ and $\gamma_3$ are connected by closed dual paths to $\partial B(2^{n})^*$
in
\[
S := B(2^{n_0})^c\cup R_1\cup R_2\cup R_3\cup R_4,
\]
and $\gamma_2$ and $\gamma_4$ are connected by open paths to $\partial B(2^{n})$ in $S$. We also observe that once $\gamma_1$, $\gamma_2$, $\gamma_3$, and $\gamma_4$ are fixed, the percolation process in
$S$ is still an independent Bernoulli percolation. Therefore, the following statement is equivalent to the statement of the lemma.
Let $E(\gamma_1,\ldots,\gamma_4)$ be the event that
(1) $\gamma_1$ and $\gamma_3$ are connected to $\partial B(2^n)^*$ by closed dual paths $P_1$ and $P_3$ in $S$, and
(2) $\gamma_2$ and $\gamma_4$ are connected to $\partial B(2^n)$ by open paths $P_2$ and $P_4$ in $S$.
Let $\widetilde E(\gamma_1,\ldots,\gamma_4)$ be the event that $E(\gamma_1,\ldots,\gamma_4)$ occurs with paths $P_1,\ldots,P_4$ satisfying the requirement that,
for all $1\leq i\leq 4$, $P_i \cap \overline{Ann(2^{n-1},2^n)} \subset \mathcal I_i(2^{n-1},2^n)$.
Then there exists a constant $C_2$ which does not depend on $n$, $n_0$, or the choice of $\gamma_i$'s, such that
\begin{equation}\label{eq:Egammas}
{\mathbb P}(E(\gamma_1,\ldots,\gamma_4))
\leq C_2
{\mathbb P}(\widetilde E(\gamma_1,\ldots,\gamma_4)).
\end{equation}

To prove the above statement we construct a family of disjoint annuli in four stages. We consider $T_1, \ldots, T_4$ on $\partial B(2^{n_0})$ as defined above.
We define
\[
l_i(1) = \min\{l~:~\exists x\in 2^l\Z^2\cap\partial B(2^{n_0})~\mbox{s.t.}~B(x,2^l)\supset T_i\}
\]
if such $l$ exists (the definition implies that it is no bigger than $n_0$),
and let $B_i(1) = B(x_i(1),2^{l_i(1)})$ be such a box. If there are several choices for the box, we pick the first one in clockwise ordering.
If there are no such $l$, we let $l_i(1) = n_0+1$ and $B_i(1) = B(2^{n_0+1})$ (in this case $x_i(1) = 0$). The boxes $B_1(1),\ldots,B_4(1)$ form a covering of $\partial B(2^{n_0})$ such that $B_i(1) \supset T_i$, and either $x_i(1)\in T_i$ or $x_i(1) = 0$ (in which case $B_i(1) = B(2^{n_0+1})$).

We also define
\[
\widetilde l_i(1) = \min\{l\geq l_i(1)~:~B(x_i(1),2^l)\supset B_j(1)~\mbox{for some}~j\neq i\}.
\]
Let
$Ann_i(1) = Ann(x_i(1); 2^{l_i(1)}, 2^{\widetilde l_i(1)-3}) = B(x_i(1),2^{\widetilde l_i(1)-3})\setminus B_i(1)$, if $\widetilde l_i(1) - 5 > l_i(1)$; otherwise, let $Ann_i(1) = \emptyset$.
Note that if $l_i(1) \geq l_{i-1}(1)$ or $l_i(1) \geq l_{i+1}(1)$, then $Ann_i(1) = \emptyset$. In particular, if $l_i(1) = n_0+1$, then $Ann_i(1) = \emptyset$. If $Ann_i(1) \neq \emptyset$, we let $\widetilde B_i(1) = B(x_i(1), 2^{\widetilde l_i(1)-3})$; otherwise, we let $\widetilde B_i(1) = B_i(1)$.
We write $\widetilde B_i(1) = B(x_i(1), 2^{l_i(1)'})$. We observe that the boxes $\widetilde B_i(1)$ form a covering of $\partial B(2^{n_0})$, and that there exists $i$ such that $|l_{i+1}(1)'-l_i(1)'| \leq 5$,
in other words, $\widetilde B_{i+1}(1)$ and $\widetilde B_i(1)$ are comparable in size.  This completes the first stage of our construction.

We proceed further by defining
\[
l_i(2) = \min\{l~:~\exists x\in 2^l\Z^2\cap\partial B(2^{n_0})~\mbox{s.t.}~B(x,2^l)\supset (\widetilde B_i(1)\cup\widetilde B_{i+1}(1))\}
\]
if such $l$ exists (it is necessarily not bigger than $n_0$),
and let $B_i(2) = B(x_i(2),2^{l_i(2)})$ be such a box (if there are several choices, we pick the first one in clockwise ordering).
If there are no such $l$, we let $l_i(2) = n_0+1$ and $B_i(2) = B(2^{n_0+1})$ (in this case $x_i(2) = 0$).

We also define
\[
\widetilde l_i(2) = \min\{l\geq l_i(2)~:~B(x_i(2),2^l)\supset \widetilde B_j(1)~\mbox{for some}~j\neq i,i+1\}.
\]
Let
$Ann_i(2) = Ann(x_i(2); 2^{l_i(2)}, 2^{\widetilde l_i(2)-3}) = B(x_i(2),2^{\widetilde l_i(2)-3})\setminus B_i(2)$, if $\widetilde l_i(2) - 5 > l_i(2)$; otherwise, let $Ann_i(2) = \emptyset$. If $Ann_i(2) \neq \emptyset$, we define $\widetilde B_i(2) = \widetilde B_{i+1}(2) = B(x_i(2), 2^{\widetilde l_i(2)-3})$.
For all remaining indices $i$ for which $\widetilde B_i(2)$ is not yet defined, we let $\widetilde B_i(2) = \widetilde B_i(1)$. In other words, if we have not succeeded in building a nonempty annulus around $\widetilde B_i(1)$, we take this box unchanged to the next stage of our construction.
We write $\widetilde B_i(2) = B(x_i(2), 2^{l_i(2)'})$. We observe that the boxes $\widetilde B_i(2)$ form a covering of $\partial B(2^{n_0})$, and that there exists $i$ such that $|l_{i+1}(2)'-l_i(2)'| \leq 5$,
$|l_{i+2}(2)'-l_i(2)'| \leq 5$ and $|l_{i+2}(2)'-l_{i+1}(2)'| \leq 5$.
In other words $\widetilde B_i(2)$, $\widetilde B_{i+1}(2)$ and $\widetilde B_{i+2}(2)$ are comparable in size.

In the third stage, for any $i$, we define
\[
l_i(3) = \min\{l~:~\exists x\in 2^l\Z^2\cap\partial B(2^{n_0})~\mbox{s.t.}~B(x,2^l)\supset (\widetilde B_i(2)\cup\widetilde B_{i+1}(2)\cup\widetilde B_{i+2}(2))\}
\]
if such $l$ exists (it is necessarily not bigger than $n_0$),
and let $B_i(3) = B(x_i(3),2^{l_i(3)})$ be such a box (if there are several choices, we pick the first one in clockwise ordering).
If there are no such $l$, we let $l_i(3) = n_0+1$ and $B_i(3) = B(2^{n_0+1})$ (in this case $x_i(3) = 0$).

We also define
\[
\widetilde l_i(3) = n_0+1.
\]

Let
$Ann_i(3) = Ann(x_i(3); 2^{l_i(3)}, 2^{\widetilde l_i(3)-3}) = B(x_i(3),2^{\widetilde l_i(3)-3})\setminus B_i(3)$, if $\widetilde l_i(3) - 5 > l_i(3)$, otherwise let $Ann_i(3) = \emptyset$. If $Ann_i(3) \neq \emptyset$, we define $\widetilde B_i(3) = \widetilde B_{i+1}(3) = \widetilde B_{i+2}(3) = B(x_i(3), 2^{\widetilde l_i(3)-3})$.
For all remaining indices $i$ for which $\widetilde B_i(3)$ is not yet defined, we let $\widetilde B_i(3) = \widetilde B_i(2)$.
In other words, if we have not succeeded in building a nonempty annulus around $\widetilde B_i(2)$, we take this box unchanged to the next stage of our construction.
We write $\widetilde B_i(3) = B(x_i(3), 2^{l_i(3)'})$. We observe that the boxes $\widetilde B_i(3)$ form a covering of $\partial B(2^{n_0})$ such that $|l_{i+1}(3)'-l_i(3)'| \leq 5$ for all $i$.
Moreover, all of these boxes are contained in $B(2^{n_0+1})$.

Finally, we define the annulus $Ann_i(4) = Ann(2^{n_0+1}, 2^{n})$.

Before we continue, we list some properties of the above constructed annuli. Let us call the annuli $(Ann_i(k))_i$ {\it level} $k$ {\it annuli}. First, $(Ann_i(1))_i$, $(Ann_j(2))_j$, $(Ann_k(3))_k$, and $(Ann_l(4))_l$ are {\it disjoint} among levels and between levels. In addition, there are at most two nonempty level $1$ annuli, at most one level $2$, $3$ or $4$ annuli each.
Next, we observe that $x_i(k)\neq 0$ if $k<4$ and $Ann_i(k)\neq \emptyset$. In other words, the only nonempty annulus centered at the origin is the level $4$ annulus. Last, the event $E(\gamma_1,\ldots,\gamma_4)$ implies the existence of crossings of annulus $Ann_i(1)$ by path $P_i$, annulus $Ann_i(2)$ by paths $P_i$ and $P_{i+1}$,
annulus $Ann_i(3)$ by paths $P_i$, $P_{i+1}$ and $P_{i+2}$, and annulus $Ann_i(4)$ by all four paths $P_1,\ldots,P_4$. Some examples of families of annuli are illustrated on Figure~\ref{fAnnuli}.

\begin{figure}
\begin{center}
\includegraphics[width=6cm]{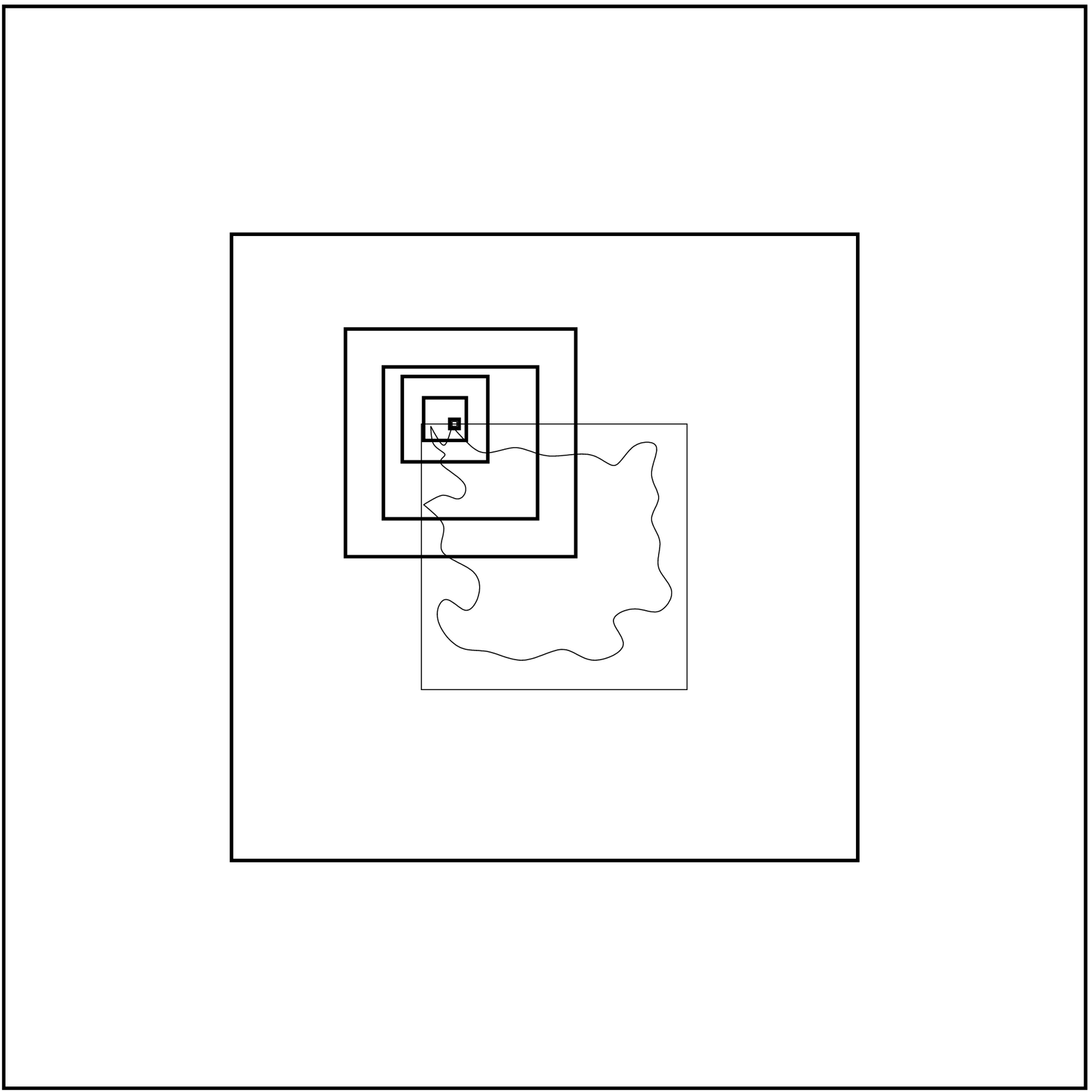}~~~~
\includegraphics[width=6cm]{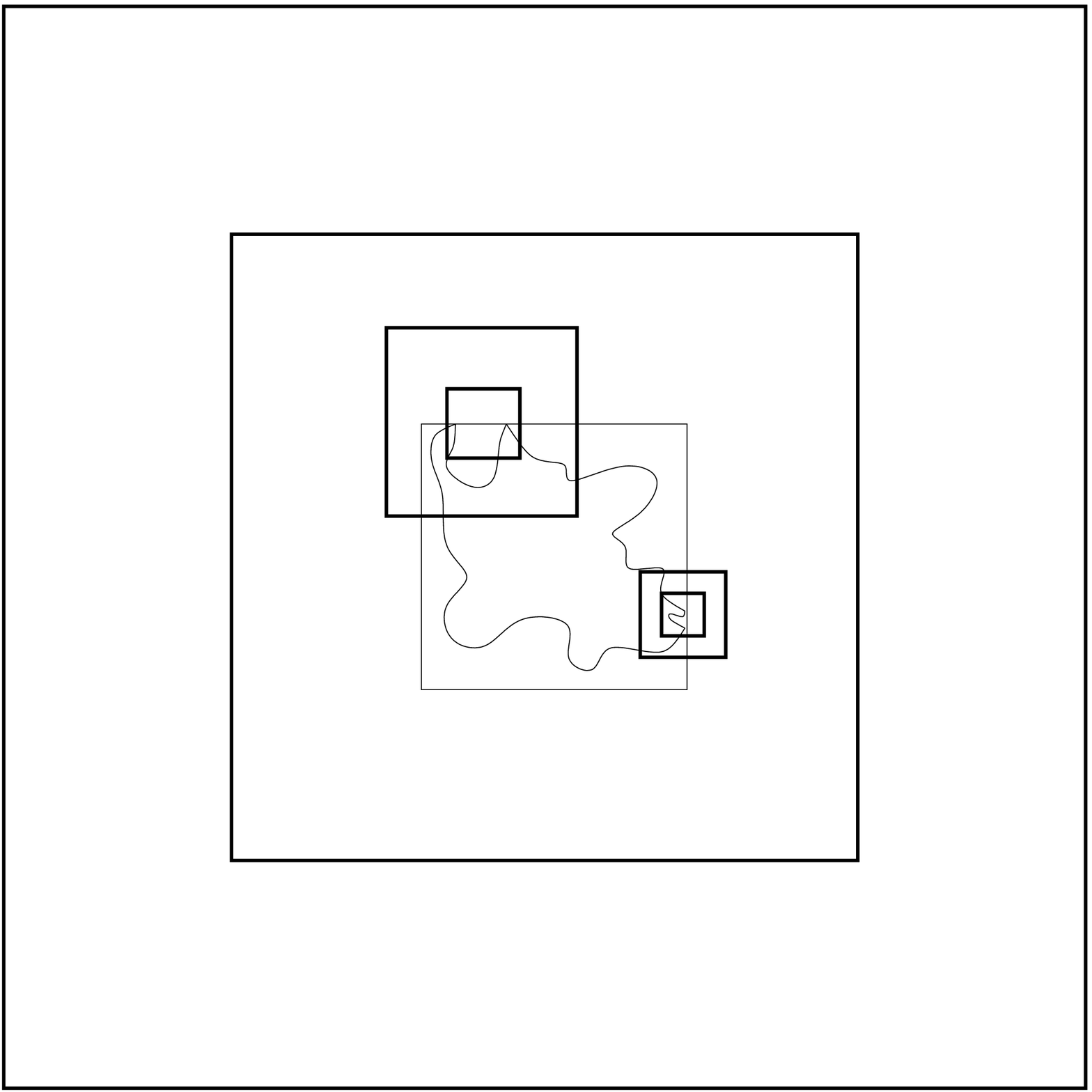}
\caption{The family of annuli in the left figure consists of one annulus of each level.
The family of annuli in the right figure consists of two level $1$ annuli and one level $4$ annulus.
In general, there are at most two non-empty level $1$ annuli, at most one level $2$, $3$ or $4$ annulus each.}
\label{fAnnuli}
\end{center}
\end{figure}

To show \eqref{eq:Egammas}, our strategy is to bound the probability of the event $E(\gamma_1,\ldots,\gamma_4)$ by the product of probabilities of crossing events in such annuli. We will then use Kesten's ideas (Lemmas~4 and 5 in \cite{kesten}) to prove an arm separation statement for each of the crossing events, and this will allow us to use the RSW Theorem (Section~11.7 in \cite{Grimmett}) and the generalized FKG inequality (Lemma~3 in \cite{kesten}) to `glue' those crossing into solid paths from the $\gamma_i$'s to $\partial B(2^{n})$ in such a way that the event $\widetilde E(\gamma_1,\ldots,\gamma_4)$ occurs.

We should be more careful though, since we have to take into account that we consider paths in $S$.
For $k<4$ and for each nonempty annulus $Ann_i(k)$, we define $a_i(k)$ as the first point on $\gamma_{i-1}$ (seen as an oriented path from $a_{i-1}$ to $a_{i}$) that belongs to $Ann_i(k)$, and
$b_i(k)$ as the last point on $\gamma_{i+k}$ (seen as an oriented path from $a_{i+k}$ to $a_{i+k+1}$) that belongs to $Ann_i(k)$.
Note that such points always exist if $Ann_i(k)\neq \emptyset$.
We then define the set $S_i(k)$ as the subset of $S$ with boundary that consists of four pieces (in clockwise order): the piece of $\partial B(x_i(k),2^{\widetilde l_i(k)-3})$ between $a_i(k)$ and $b_i(k)$, the piece of $\gamma_{i+k}$ from $b_i(k)$ to the last intersection of $\gamma_{i+k}$ with
$\partial B(x_i(k),2^{l_i(k)})$, the piece of $\partial B(x_i(k),2^{l_i(k)})$ in $S$, and the piece of $\gamma_{i-1}$ from $a_i(k)$ to the last intersection with $\partial B(x_i(k),2^{l_i(k)})$.
Let $L_i(k)$ be the common piece of the boundary of $S_i(k)$ and $\partial B(x_i(k),2^{\widetilde l_i(k)-3})$ between $a_i(k)$ and $b_i(k)$.

\begin{figure}
\begin{center}
\includegraphics[width=12cm]{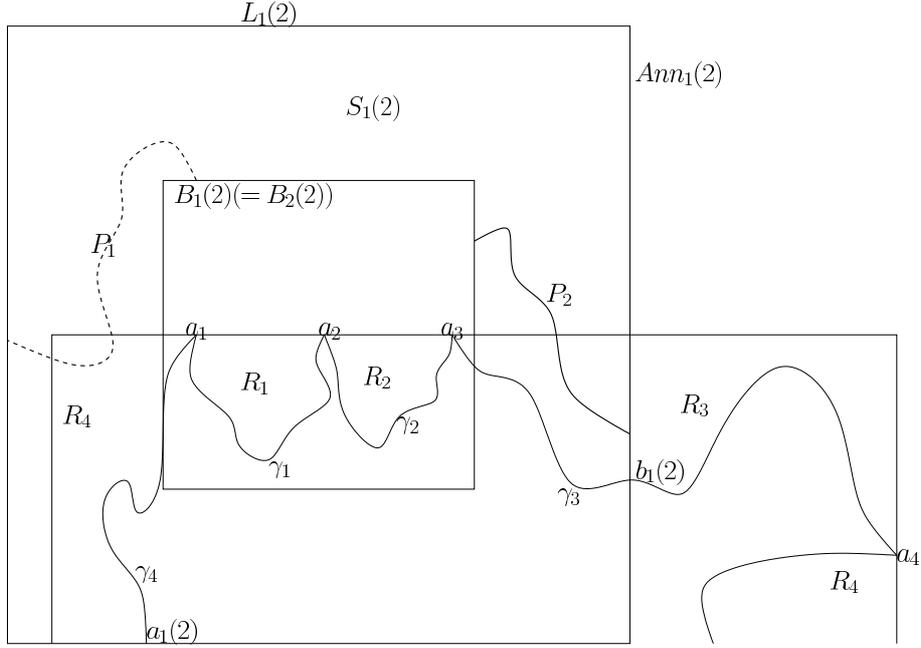}
\caption{The small box is $B(x_1(2),2^{l_1(2)})$ and the big box is $\partial B(x_1(2),2^{\widetilde l_1(2)-3})$. The event $E_1(2)$ occurs if there is a closed dual path $P_1$ and an open path $P_2$ from $B(x_1(2),2^{l_1(2)})$ to $L_1(2)$ in $S_1(2)$.}
\label{fEventSik}
\end{center}
\end{figure}

For $k<4$ and for each nonempty annulus $Ann_i(k)$, let $E_i(k)$ be the event that there exist $k$ disjoint paths from $B(x_i(k),2^{l_i(k)})$ to $L_i(k)$ in $S_i(k)$ such that the order and the status of paths (occupied or vacant) are induced by the order and the status of $P_i,\ldots,P_{i+k-1}$.
For $k<4$ and for each empty annulus $Ann_i(k)$, let $E_i(k)$ be the sure event.
Let $E_i(4)$ be the event that the annulus $Ann_i(4)$ is crossed by two open and two closed dual paths such that the open paths are separated by the closed paths.
Since the sets $S_i(k)$ and $Ann_i(4)$ are disjoint, we obtain
\[
{\mathbb P}(E(\gamma_1,\ldots,\gamma_4))
\leq
\prod_{i,k=1}^4 {\mathbb P}(E_i(k)).
\]
For every nonempty annulus $Ann_i(k)$ we define landing areas $I_i(k) = \{I_i^{(1)}(k),\ldots,I_i^{(k)}(k)\}$ on $Ann(x_i(k); 2^{l_i(k)}, 2^{l_i(k)+1})$ and $\widetilde I_i(k) = \{\widetilde I_i^{(1)}(k),\ldots,\widetilde I_i^{(k)}(k)\}$ on
$Ann(x_i(k); 2^{\widetilde l_i(k) - 4}, 2^{\widetilde l_i(k) - 3})$ as follows. If $x_i(k) = 0$, we let
$I_i^{(1)}(k) = [-2^{n_0+1},2^{n_0+1}]\times[2^{n_0+1},2^{n_0+2}]$,
$I_i^{(2)}(k) = [2^{n_0+1},2^{n_0+2}]\times[-2^{n_0+1},2^{n_0+1}]$,
$I_i^{(3)}(k) = [-2^{n_0+1},2^{n_0+1}]\times[-2^{n_0+2},-2^{n_0+1}]$,
$I_i^{(4)}(k) = [-2^{n_0+2},-2^{n_0+1}]\times[-2^{n_0+1},2^{n_0+1}]$ (note that we only need to introduce such landing areas for $k=4$, since otherwise $Ann_i(k) = \emptyset$ when $x_i(k)=0$).
If $x_i(k) \neq 0$, we consider a side of $\partial B(x_i(k), 2^{l_i(k)})$ which is parallel to the side of $\partial B(2^{n_0})$ containing $x_i(k)$ and does not intersect $B(2^{n_0})$. If several sides satisfy these conditions, we pick the first one in clockwise ordering.
We partition this side into four pieces of equal length $2^{l_i(k)-1}$, and we consider four rectangles in $Ann(x_i(k); 2^{l_i(k)}, 2^{l_i(k)+1})$ of size
$2^{l_i(k) - 1} \times 2^{l_i(k)}$ such that each of the above defined pieces of $\partial B(x_i(k), 2^{l_i(k)})$ is a side of one of the rectangles.
We then define $I_i(k)$ to be the set of first $k$ (in clockwise ordering) of these rectangles.

If $x_i(k) = 0$, we let
$\widetilde I_i^{(1)}(k) = [-2^{\widetilde l_i(k) - 4},2^{\widetilde l_i(k) - 4}]\times[2^{\widetilde l_i(k) - 4},2^{\widetilde l_i(k) - 3}]$,
$\widetilde I_i^{(2)}(k) = [2^{\widetilde l_i(k) - 4},2^{\widetilde l_i(k) - 3}]\times[-2^{\widetilde l_i(k) - 4},2^{\widetilde l_i(k) - 4}]$,
$\widetilde I_i^{(3)}(k) = [-2^{\widetilde l_i(k) - 4},2^{\widetilde l_i(k) - 4}]\times[-2^{\widetilde l_i(k) - 3},-2^{\widetilde l_i(k) - 4}]$,
$\widetilde I_i^{(4)}(k) = [-2^{\widetilde l_i(k) - 3},-2^{\widetilde l_i(k) - 4}]\times[-2^{\widetilde l_i(k) - 4},2^{\widetilde l_i(k) - 4}]$ (again, the only useful case for us here is $k=4$, since otherwise $Ann_i(k)=\emptyset$).
If $x_i(k) \neq 0$, we consider a side of $\partial B(x_i(k), 2^{\widetilde l_i(k) - 4})$ which is parallel to the side of $\partial B(2^{n_0})$ containing $x_i(k)$ and does not intersect $B(2^{n_0})$. If several sides satisfy these conditions, we pick the first one in clockwise ordering.
We partition this side into four pieces of equal length $2^{\widetilde l_i(k) - 5}$, and we consider four rectangles in
$Ann(x_i(k); 2^{\widetilde l_i(k) - 4}, 2^{\widetilde l_i(k) - 3})$ of size
$2^{\widetilde l_i(k) - 5} \times 2^{\widetilde l_i(k) - 4}$ such that each of the above defined pieces of $\partial B(x_i(k), 2^{\widetilde l_i(k) - 4})$ is a side of one of the rectangles.
We then define $\widetilde I_i(k)$ to be the set of first $k$ (in clockwise ordering) of these rectangles.

Let $\widetilde E_i(k)$ be the event that $E_i(k)$ occurs, the intersection of $j$th path with $Ann(x_i(k); 2^{l_i(k)}, 2^{l_i(k)+1})$ is contained in $I_i^{(j)}(k)$, and the intersection of $j$th path with $Ann(x_i(k); 2^{\widetilde l_i(k) - 4}, 2^{\widetilde l_i(k) - 3})$ is contained in $\widetilde I_i^{(j)}(k)$.
We show that there exists a constant $C_5$ such that for any choice of $\gamma_i$'s,
\begin{equation}\label{eqArmSeparationAnnuli}
{\mathbb P}(E_i(k))\leq C_5 {\mathbb P}(\widetilde E_i(k)).
\end{equation}
\begin{proposition}
The above statement implies Lemma~\ref{lKesten}.
\end{proposition}
\begin{proof}
The proof of this proposition is similar to the proof of (2.43) in \cite{kesten}. It is based on the RSW theorem (Section~11.7 in \cite{Grimmett}) and the generalized FKG inequality (Lemma~3, \cite{kesten}). From the construction of $(Ann_i(k))_{i,k}$ it follows that one can define disjoint (wide enough) regions $Q_1,\ldots,Q_4$ in the complement of $\cup_{i,k}Ann_i(k)$ such that the probability of the event that\\
(1) $\widetilde E_i(k)$ occur for all $i$ and $k$, and\\
(2) the $j$th crossings of $Ann_i(k)$ are connected through $Q_j$ into a single path from $\gamma_j$ to $\partial B(2^n)$\\
is bounded from below by $C_6 \prod_{i,k=1}^4 {\mathbb P}(\widetilde E_i(k))$, with a constant $C_6$ which does not depend on $n_0$, $n$, $\gamma_i$'s.
\end{proof}
We now prove inequality \eqref{eqArmSeparationAnnuli}. We observe that the case $k=4$ follows from Lemmas~4 and 5 in \cite{kesten}. We also note that the case $k=1$ follows from the RSW theorem and the FKG inequality. Therefore it is sufficient to consider cases $k=2$ and $k=3$. We only consider here the case $k=2$. The proof of the other case is similar.

If $Ann_i(2) = \emptyset$, then there is nothing to prove, so we assume that $Ann_i(2) \neq \emptyset$. Recall that in this case
$Ann_i(2) = Ann(x_i(2); 2^{l_i(2)}, 2^{\widetilde l_i(2)-3})$ with $\widetilde l_i(2)-5> l_i(2)$.
We consider the event $E_i(k)$ that there exist paths $\bar P_i$ and $\bar P_{i+1}$ from $B(x_i(k),2^{l_i(k)})$ to $L_i(k)$ in $S_i(k)$. Without loss of generality we can assume that $\bar P_i$ is open and $\bar P_{i+1}$ is closed.

\begin{figure}
\begin{center}
\includegraphics[width=12cm]{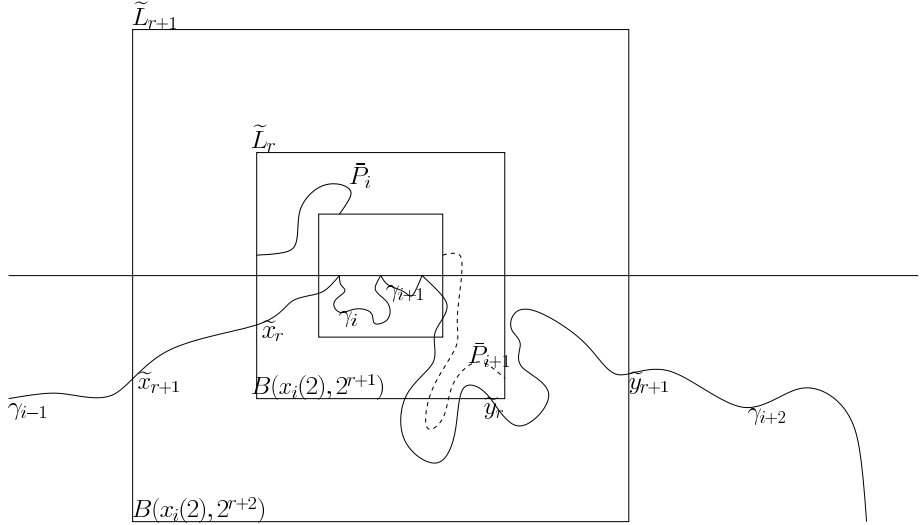}
\caption{It is not enough to have well-separated paths in $\widetilde S_r$. There should be enough space in $\widetilde S_{r+1}\setminus\widetilde S_r$ to extend those paths. In this figure, path $\bar P_i$ can be easily extended from $\widetilde L_r$ to $\widetilde L_{r+1}$ in $\widetilde S_{r+1}\setminus\widetilde S_r$, but not the path $\bar P_{i+1}$. To overcome this difficulty we introduce events $D_r$.}
\label{fEventTrap}
\end{center}
\end{figure}

Fix $r\in [l_i(2), \widetilde l_i(2) - 5)$. We define $\widetilde x_r$ as the first point on $\gamma_{i-1}$ (seen as an oriented path from $a_{i-1}$ to $a_i$) which is contained in $\partial B(x_i(2),2^{r+1})$, and $\widetilde y_r$ as the last point on $\gamma_{i+2}$ (seen as an oriented path from $a_{i+2}$ to $a_{i+3}$) that is contained in $\partial B(x_i(2),2^{r+1})$.
Note that such points always exist.
We define the set $\widetilde S_r$ as the subset of $S$ with boundary that consists of four pieces (in clockwise order): the piece of $\partial B(x_i(2),2^{r+1})$ between $\widetilde x_r$ and $\widetilde y_r$, the piece of $\gamma_{i+2}$ from $\widetilde y_r$ to the last intersection of $\gamma_{i+2}$ with
$\partial B(x_i(2),2^{l_i(2)})$, the piece of $\partial B(x_i(2),2^{l_i(2)})$ in $S$, and the piece of $\gamma_{i-1}$ from $\widetilde x_r$ to the last intersection with $\partial B(x_i(2),2^{l_i(2)})$.
Let $\widetilde L_r$ be the common piece of the boundary of $\widetilde S_r$ and $\partial B(x_i(2),2^{r+1})$ between $\widetilde x_r$ and $\widetilde y_r$.

Consider the event $E_r$ that there exists an open path and a closed dual path from $B(x_i(2),2^{l_i(2)})$ to $\widetilde L_{r+2}$ in $\widetilde S_{r+2}$ (their order is induced by the order of $\bar P_i$ and $\bar P_{i+1}$). We define the landing areas $\widetilde I_r^{(j)}$ on $Ann(x_i(2); 2^{r},2^{r+1})$ in the same way that we defined
the landing areas $I_i(k) = \{I_i^{(1)}(k),\ldots,I_i^{(k)}(k)\}$ on $Ann(x_i(k); 2^{l_i(k)}, 2^{l_i(k)+1})$. Let $\widetilde E_r$ be the event that $E_r$ occurs, the intersection of $j$th path with $Ann(x_i(2); 2^{r+2},2^{r+3})$ is contained in $\widetilde I_{r+2}^{(j)}$.

\begin{figure}
\begin{center}
\includegraphics[width=12cm]{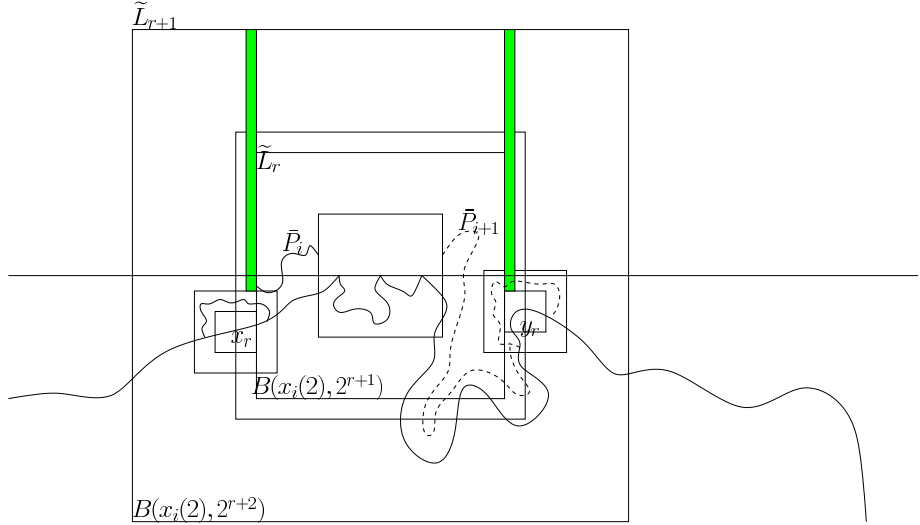}
\caption{If events $E_r$ and $D_r$ occur, the paths $\bar P_i$ and $\bar P_{i+1}$ can be extended from $\widetilde L_r$ to $\widetilde L_{r+1}$ through wide enough (green) corridors.}
\label{fEventDr}
\end{center}
\end{figure}

In essentially the same way as in the proof of Lemma~4 in \cite{kesten} (except that we need to consider events $D_r$, defined below), we
show that there exists a constant $C_7$ that does not depend on $r$ and $\gamma_i$'s such that
\begin{equation}\label{eq:peeling1}
{\mathbb P}(E_r) \leq \delta {\mathbb P}(E_{r-3}) + C_7 {\mathbb P}(\widetilde E_r).
\end{equation}
To prove this inequality, we first make some definitions. Let $\eta$ be a small positive number. We define $x_r$ to be the point on $\gamma_{i-1}\cap Ann(x_i(2); 2^{r+1},2^{r+1}+\eta 2^r)$ such that if we rotate the set
$\gamma_{i-1}\cap Ann(x_i(2); 2^{r+1},2^{r+1}+\eta 2^r)$ about $x_i(2)$ clockwise then the first point of this set which comes into contact with $\partial B(2^{n_0})$ is $x_r$.
If there are several such points, we choose one arbitrarily.
Similarly, let $y_r$ be the point on $\gamma_{i+2}\cap Ann(x_i(2); 2^{r+1},2^{r+1}+\eta 2^r)$ such that if we rotate the set
$\gamma_{i+2}\cap Ann(x_i(2); 2^{r+1},2^{r+1}+\eta 2^r)$ about $x_i(2)$ counterclockwise then the first point of this set which comes into contact with $\partial B(2^{n_0})$ is $y_r$.
If there are several such points, we choose one arbitrarily.
Let $D_r$ be the event that there is an open path in $Ann(x_r;\eta 2^r,\sqrt\eta 2^r)$ separating $x_r$ from $\partial B(2^n)$ in $S$ (one can think of this path as a connected piece in $S\cap Ann(x_r;\eta 2^r,\sqrt\eta 2^r)$ of a complete circuit in $Ann(x_r;\eta 2^r,\sqrt\eta 2^r)$ around $x_r$), and that there is a closed dual path in $Ann(y_r;\eta 2^r,\sqrt\eta 2^r)$ separating $y_r$ from $\partial B(2^n)$ in $S$.
Note that for any $\delta>0$, we can choose $\eta$ such that the probability of $D_r$ is at least $1-\delta/4$.

We now examine the event $D_r$. If $D_r$ occurs, we define $\bar \gamma_r$ as $\gamma_{i-1}$ except that we replace the piece of $\gamma_{i-1}$ between the first and the last intersections of $\gamma_{i-1}$ with the innermost open path in $Ann(x_r;\eta 2^r,\sqrt\eta 2^r)$ (from the definition of $D_r$) with the piece of this innermost path between the intersection points.
Similarly we define $\bar \gamma_r'$ as $\gamma_{i+2}$ except that we replace the piece of $\gamma_{i+2}$ between the first and the last intersections of $\gamma_{i+2}$ with the innermost closed path in $Ann(y_r;\eta 2^r,\sqrt\eta 2^r)$ (from the definition of $D_r$) with the piece of this innermost path between the intersection points.
We define the set $\overline {S_r}$ in the same way as $\widetilde S_r$ except that we use $\bar \gamma_r$ and $\bar \gamma_r'$ instead of $\gamma_{i-1}$ and $\gamma_{i+2}$. In particular, in Figure~\ref{fEventDr}, the region $\widetilde S_r$ does not contain the ``bubble'' formed by $\gamma_{i+2}$ near the bottom right corner of $\partial B(x_i(2),2^{r+1})$, but $\overline {S_r}$ does.

To show \eqref{eq:peeling1}, we consider three different cases: (1) $E_{r-3}$ occurs and $D_{r}$ does not occur;
(2) events $E_{r-3}$ and $D_{r}$ occur, but the open and closed paths (from the definition of $E_r$) from $B(2^{r})$ to $\widetilde L_{r}$ in the region $\overline {S_{r}}$ are not well-separated on $\widetilde L_{r}$ (see the proof of Lemma~4 in \cite{kesten} for the precise definition of well-separated paths);
(3) events $E_{r-3}$ and $D_{r}$ occur, and the open and closed paths (from the definition of $E_r$) from $B(2^{r})$ to $\widetilde L_{r}$ in the region $\overline {S_{r}}$ are well-separated on $\widetilde L_r$ (see the proof of Lemma~4 in \cite{kesten} for the precise definition). We bound the probability of the first two events by $\delta {\mathbb P}(E_{r-3})$. In order to bound the probability of the third event by $C_7 {\mathbb P}(\widetilde E_r)$, we first condition on the innermost open path in $Ann(x_r;\eta 2^r,\sqrt\eta 2^r)$ and on the innermost closed path in $Ann(y_r;\eta 2^r,\sqrt\eta 2^r)$ defined by $D_r$; we then repeat Kesten's proof of (2.42).
We refer the interested reader to the proof of Lemma~4 in \cite{kesten} for more details.

Along with \eqref{eq:peeling1}, we use one more inequality.  It is an application of the RSW theorem and the generalized FKG inequality (Lemma~3 in \cite{kesten}): there exists a constant $C_8$ that does not depend on $r$ or the $\gamma_i$'s such that
\begin{equation}\label{eq:peeling2}
{\mathbb P}(\widetilde E_r) \geq C_8 {\mathbb P}(\widetilde E_{r-3}).
\end{equation}
Inequalities \eqref{eq:peeling1} and \eqref{eq:peeling2} imply that there exists a constant $C_9$ that does not depend on $r$ or on the $\gamma_i$'s such that
${\mathbb P}(E_r) \leq C_9 {\mathbb P}(\widetilde E_r)$.
This inequality proves arm separation on the outer boundary of the annulus. Similar ideas apply to obtain an arm separation result on the inner boundary of the same annulus (see Lemma~5 in \cite{kesten}) and to prove (\ref{eqArmSeparationAnnuli}).

\end{proof}

\section{Proof of Theorem~\ref{thm4IIC}}\label{4IICsec}
\noindent

We prove only the first statement; the proof of the second is similar (see, e.g., the proof of the second statement in \cite[Theorem~3]{Jarai}).  We follow the same method used in \cite[Theorem 3]{Jarai} but because many difficulties arise, we present details of the entire proof.  Pick an edge $e = \langle e_x,e_y \rangle$, and let $n = 2|e|/3$ (this choice makes $e$ in the middle of $Ann(n,2n)$).
Let $\epsilon > 0$.

\bigskip
\textit{Step 1}.  First we give a lower bound for the probability that $e \in {\cal O}$.
Recall the definition of $p_n$ from Section~\ref{secCL} and the definition of $p_n(k)$ from (\ref{pdefgen}).
The constant $C_*$ will be determined later.  Consider the event $D_e$ that

\begin{enumerate}

\item there exist $p_c$-open circuits around the origin in the annuli $Ann(n/2,n)$ and $Ann(2n,4n)$;

\item there exist two disjoint $p_c$-open paths, one connecting $y$ to the circuit  in $Ann(2n,4n)$ and one connecting $x$ to the circuit in $Ann(n/2,n)$;

\item there exists a $(2p_n-p_c)$-closed dual circuit with one defect around 0 in the annulus $Ann(n,2n)^*$ which includes the edge $e^*$ as its defect;

\item $\tau_e \in [p_n,2p_n-p_c)$; and

\item the $p_c$-open circuit in $Ann(2n,4n)$ is connected to $\infty$ by a $p_n$-open path.

\end{enumerate}
An illustration of this event is in Figure \ref{Defig}.

\begin{figure}
\begin{center}
\scalebox{1}{\includegraphics[viewport = 0in 0in 3in 2.5in]{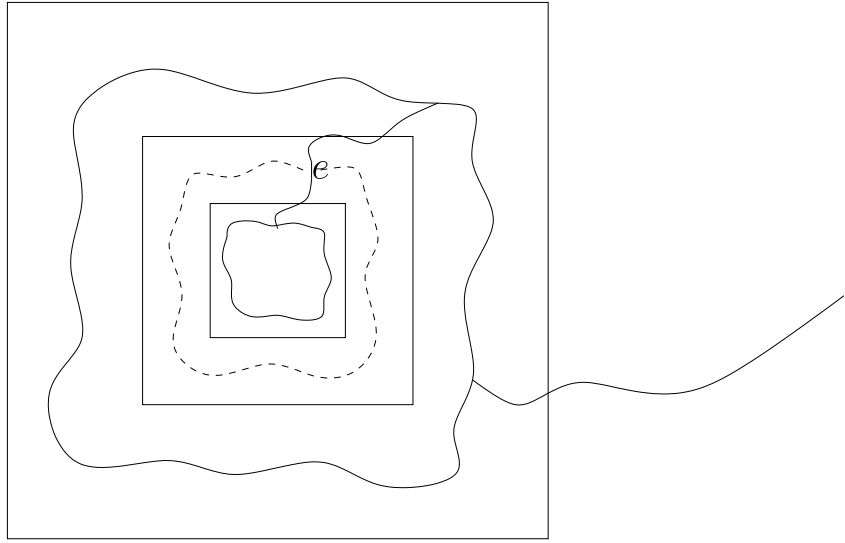}}
\end{center}
\label{Defig}
\caption{The event $D_e$.  The solid curves represent occupied paths and the dotted curves represent vacant paths.  The outer circuit is connected to $\infty$ by a $p_n$-open path.  The boxes, from smallest to largest, are $B(n), B(2n),$ and $B(4n)$.  The edge $e$ is an outlet of the invasion.}
\end{figure}
By RSW arguments, \cite[Lemma~6.3]{DSV} and the fact that $L(2p_n-p_c)$ is comparable with $n$ (see e.g. \cite[(4.35)]{kesten}),
\begin{equation}
{\P}(D_e) \asymp (p_n - p_c) {\P}_{cr}(A_n^{2,2}),
\end{equation}
where $A_n^{2,2}$ is the event that the edge $e-e_x$ (we recall this notation means $\langle 0, e_y-e_x \rangle$) is connected to $\partial B(n)$ by two disjoint $p_c$-open paths and $(e-e_x)^*$ is connected to $\partial B(n)^*$ by two disjoint $p_c$-closed dual paths such that the open and closed paths alternate.  Since $D_e$ implies that $e \in {\cal O}$, we have for all $e$,

\begin{equation}
\label{backbonelowerbdeq}
{\P}(e \in {\cal O}) \geq C_1 (p_n-p_c){\P}_{cr}(A_n^{2,2}).
\end{equation}

\bigskip
\textit{Step 2}.  Recall the definition of an open circuit with defects from Section~\ref{IICsec}.  We will now show that the event
\[ A_{N,M}(e_x, p_c) = \{ \textrm{there is a } p_c \textrm{-open circuit with 2 defects around } e_x \textrm{ in } Ann(e_x, N,M) \} \]
has ${\P}(A_{N,M}(e_x, p_c), \theta_eE ~|~e \in {\cal O})$ close to ${\P}(\theta_eE 
~|~ e \in {\cal O})$ for certain values of $N<M$.
To this end, recall the definition of the event $H_{n,k}$ in (\ref{Hnkdefeq}) and write $H$ for the event $H_{n,1}$.  By (\ref{expdecay}) and (\ref{backbonelowerbdeq}), we can choose $C_*$ independent of $n$ such that

\begin{equation}
\label{approx1eq}
{\P}(\theta_eE, H^c ~|~ e \in {\cal O}) < \epsilon.
\end{equation}
When the event $H$ occurs, the invasion enters the $p_n(1)$-open infinite cluster \textit{before} it reaches $e$.  Hence if $e$ is an outlet, then $e$ must be connected to $\partial B(e_x,n/4)$ by two disjoint $p_n(1)$-open paths and $e^*$ must be connected to $\partial B(e_x,n/4)^*$ by two disjoint $p_c$-closed dual paths such that the open and closed paths alternate and are all disjoint.  Also, the weight $\tau_e$ must be in the interval $[p_c,p_n(1)]$.  If, in addition, the event $A_{N,M}(e_x,p_c)$ does not occur, then there must be yet another $p_c$-closed dual path from $B(e_x,N)^*$ to $\partial B(e_x,M)^*$.  This crossing has the property that it is disjoint from the two $p_c$-closed paths which are already present; however, it does not need to be disjoint from the $p_n(1)$-open crossings.  Therefore,

\[ {\P}(\theta_eE, H, A_{N,M}(e_x,p_c)^c, e \in {\cal O}) \]
\[ \leq C_2(p_n(1)-p_c){\P}(A_n^{2,2}(p_n(1),p_c)){\P}(A_{N,M}^{2,3*}(p_n(1),p_c) ~|~ A_{N,M}^{2,2}(p_n(1),p_c)), \]
where $A_{N,M}^{2,2}(p,q)$ denotes the event that $B(N)$ is connected to $\partial B(M)$ by two $p$-open paths and that $B(N)^*$ is connected to $\partial B(M)^*$ by two $q$-closed paths so that the open and closed paths alternate and are all disjoint.  The symbol $A_{N,M}^{2,3*}(p,q)$ signifies the event that $A_{N,M}^{2,2}(p,q)$ occurs but that there is an additional $q$-closed path connecting $B(N)$ to $\partial B(M)$ which is disjoint from the two other $q$-closed paths but not necessarily from the two $p$-open paths.  The above inequality, along with the estimate (\ref{backbonelowerbdeq}), gives

\[ {\P}(\theta_eE, H, A_{N,M}(e_x,p_c)^c ~|~ e \in {\cal O}) \]
\[ \leq \frac{C_2 (p_n(1)-p_c){\P}(A_n^{2,2}(p_n(1),p_c))}{C_1 (p_n-p_c){\P}_{cr}(A_n^{2,2})} {\P}(A_{N,M}^{2,3*}(p_n(1),p_c) ~|~ A_{N,M}^{2,2}(p_n(1),p_c)). \]
From (\ref{ineqKesten}) and Lemma~6.3 in \cite{DSV}, we can deduce

\begin{equation}
\label{bigcompeq}
\frac{C_2 (p_n(1)-p_c){\P}(A_n^{2,2}(p_n(1),p_c))}{C_1 (p_n-p_c){\P}_{cr}(A_n^{2,2})} \leq C_3 (C_* \log n)^2,
\end{equation}
so that

\[ {\P}(\theta_eE, H, A_{N,M}(e_x,p_c)^c ~|~ e \in {\cal O}) \]
\begin{equation}
\label{firstQeq}
\leq C_3 (C_*\log n)^2 {\P}(A_{N,M}^{2,3*}(p_n(1),p_c) ~|~ A_{N,M}^{2,2}(p_n(1),p_c)).
\end{equation}
The above can be made less than $\epsilon$ provided that $M/N$ grows fast enough with $n$.  Let us assume this for the moment; we shall choose precise values for $M$ and $N$ at the end of the proof.  Therefore, using (\ref{approx1eq}), we have

\begin{equation}
\label{approxeq1}
|{\P}(\theta_eE ~|~ e \in {\cal O}) - {\P}(\theta_eE, H, A_{N,M}(e_x,p_c) ~|~ e \in {\cal O})| < 2\epsilon.
\end{equation}

\bigskip
\textit{Step 3}.  We now condition on the outermost $p_c$-open circuit with 2 defects in $Ann(e_x,N,M)$.  For any circuit ${\cal C}$ with 2 defects around the origin in the annulus $Ann(N,M)$, let $D({\cal C})$ be the event that it is the outermost $p_c$-open circuit with 2 defects.  Notice that $D({\cal C})$ depends only on the state of edges on or outside ${\cal C}$.
For distinct ${\cal C}$, ${\cal C}'$ (i.e. the sets of edges in ${\cal C}$ and ${\cal C}'$ are different or the sets of edges in ${\cal C}$ and ${\cal C}'$ are the same but the defects are different), the events $D({\cal C}), D({\cal C}')$ are disjoint.  Therefore, the second term of (\ref{approxeq1}) is equal to

\begin{equation}
\label{circuitconditioneq}
\frac{1}{{\P}(e \in {\cal O})}\sum_{{\cal C} \subset Ann(N,M)} {\P}(\theta_eE, H, \theta_eD({\cal C}), e \in {\cal O}),
\end{equation}
where it is implied that in the sum, and in future sums like it, we only use circuits which enclose the origin.

\bigskip
\textit{Step 4}.  Let
\[ Q(\theta_e{\cal C}) = \{ \textrm{there exists } f \neq e \textrm{ interior to } \theta_e{\cal C} \textrm{ with } \tau_f \in [p_c, p_n(1)] \}. \]
We will now show that with high probability, the event $Q(\theta_e{\cal C})$ does not occur.  In other words, we will bound the probability of the event $\{H, Q(\theta_e{\cal C}), \theta_eD({\cal C}),e \in {\cal O}\}$.  Supposing that this event occurs, then both $\tau_e \in [p_c,p_n(1))$ and $A_{M,n}^{2,2}(p_n(1),p_c)$ must occur.  Notice that the events $A_{M,n}^{2,2}(p_n(1),p_c)$, $\theta_eD({\cal C})$, $\{\tau_e \in [p_c,p_n(1)) \}$, and $Q(\theta_e{\cal C})$ are all independent.   Hence ${\P}(H,Q(\theta_e{\cal C}), \theta_eD({\cal C}),e \in {\cal O})$ is at most

\[ {\P}(A_{M,n}^{2,2}(p_n(1),p_c)) {\P}(\theta_eD({\cal C})) {\P}(Q(\theta_e{\cal C})) {\P}(\tau_e \in [p_c,p_n(1))) \]
\[ \leq C_4 M^2 \frac{(p_n(1) - p_c)^2}{{\P}(A_M^{2,2}(p_n(1),p_c))} {\P}(A_n^{2,2}(p_n(1),p_c)) {\P}(\theta_eD({\cal C})),\]
where in the last inequality we used Corollary~6.1 from \cite{DSV}.
Consequently,

\[ {\P}(H,Q(\theta_e{\cal C}), \theta_eD({\cal C}) | e \in {\cal O}) \leq \left[ \frac{C_4 M^2 (p_n(1) - p_c)}{{\P}_{p_n(1)}(A_M^{2,2})} \right] \left[ \frac{(p_n(1)-p_c){\P}(A_n^{2,2}(p_n(1),p_c))}{C_1 (p_n-p_c){\P}_{cr}(A_n^{2,2})} \right] {\P}(\theta_eD({\cal C})),\]
which, by (\ref{bigcompeq}), is at most

\[ \frac{C_5 (C_* \log n)^2 M^2}{{\P}_{p_n(1)}(A_M^{2,2})} (p_n(1)-p_c){\P}(\theta_eD({\cal C})). \]
As long as $M$ is not too big, from ${\P}_{p_n(1)}(A_M^{2,2}) \asymp {\P}_{cr}(A_M^{2,2}) \geq c M^{-2}$ (see, e.g., Theorem~24 and Theorem~27 in \cite{Nolin}), we get an upper bound of

\begin{equation}
\label{fourthQeq}
C_6 (C_* \log n)^2 M^4 (p_n(1) - p_c) {\P}(\theta_eD({\cal C})) < \epsilon{\P}(\theta_eD({\cal C})).
\end{equation}
We will be able to choose such an $M$ (in fact it will be of the order of a power of $\log n$), but we delay justification of this to the end of the proof.  We henceforth assume that ${\P}(\theta_eE ~|~ e \in {\cal O})$ is within $3\epsilon$ of

\begin{equation}
\label{Qconditioneq}
\frac{1}{{\P}(e \in {\cal O})}\sum_{{\cal C} \subset Ann(N,M)} {\P}(\theta_eE, H, \theta_eD({\cal C}), Q(\theta_e {\cal C})^c, e \in {\cal O}).
\end{equation}

\bigskip
\textit{Step 5}.  We write our configuration $\omega$ as $\eta \oplus \xi$, where $\eta$ is the configuration outside or on $\theta_e{\cal C}$ and $\xi$ is the configuration inside $\theta_e{\cal C}$.  We condition on both $\eta$ and $\tau_e$: the summand of the numerator in (\ref{Qconditioneq}) becomes

\begin{equation}
\label{tauconditioneq}
{\E}\left[ {\P}(\theta_eE, H, \theta_eD({\cal C}), Q(\theta_e{\cal C})^c, e \in {\cal O} ~|~ \tau_e, \eta) \right].
\end{equation}

Call the defected dual edges in $\theta_e{\cal C}$ $e_1^*$ and $e_2^*$.  Given the value of $\tau_e$, on the event $\theta_eD({\cal C}) \cap H \cap Q(\theta_e{\cal C})^c$, the event $\{e \in {\cal O}\}$ occurs if and only if all of the following occur:

\begin{enumerate}

\item $e$ is connected to $\theta_e{\cal C} \setminus \{e_1^*,e_2^*\}$ in the interior of $\theta_e{\cal C}$ by two disjoint $p_c$-open paths;

\item $e^*$ is connected to $\{e_1^*,e_2^*\}$ in the interior of $\theta_e{\cal C}$ by two disjoint $p_c$-closed dual paths so that the $p_c$-closed paths and the $p_c$-open paths from item 1 alternate and are disjoint;

\item outside of $\theta_e{\cal C}$, $\theta_e{\cal C}$ is connected by a $\tau_e$-open path to $\infty$;

\item $\tau_e \in [p_c,p_n(1))$; and

\item there exists a $\tau_e$-closed dual path $P$ outside of $\theta_e{\cal C}$, connecting $e_1^*$ to $e_2^*$ (both of which are $\tau_e$-closed) with the following properties:

\begin{enumerate}

\item $P \cup B(e_x,M)^*$ contains a circuit around the origin; and

\item the invasion graph contains a vertex from ${\cal C}$ before it contains an edge $f$ with $f^*$ from $P$.

\end{enumerate}

\end{enumerate}

\begin{figure}
\begin{center}
\scalebox{1}{\includegraphics*[viewport = 0in 0in 5in 3.3in]{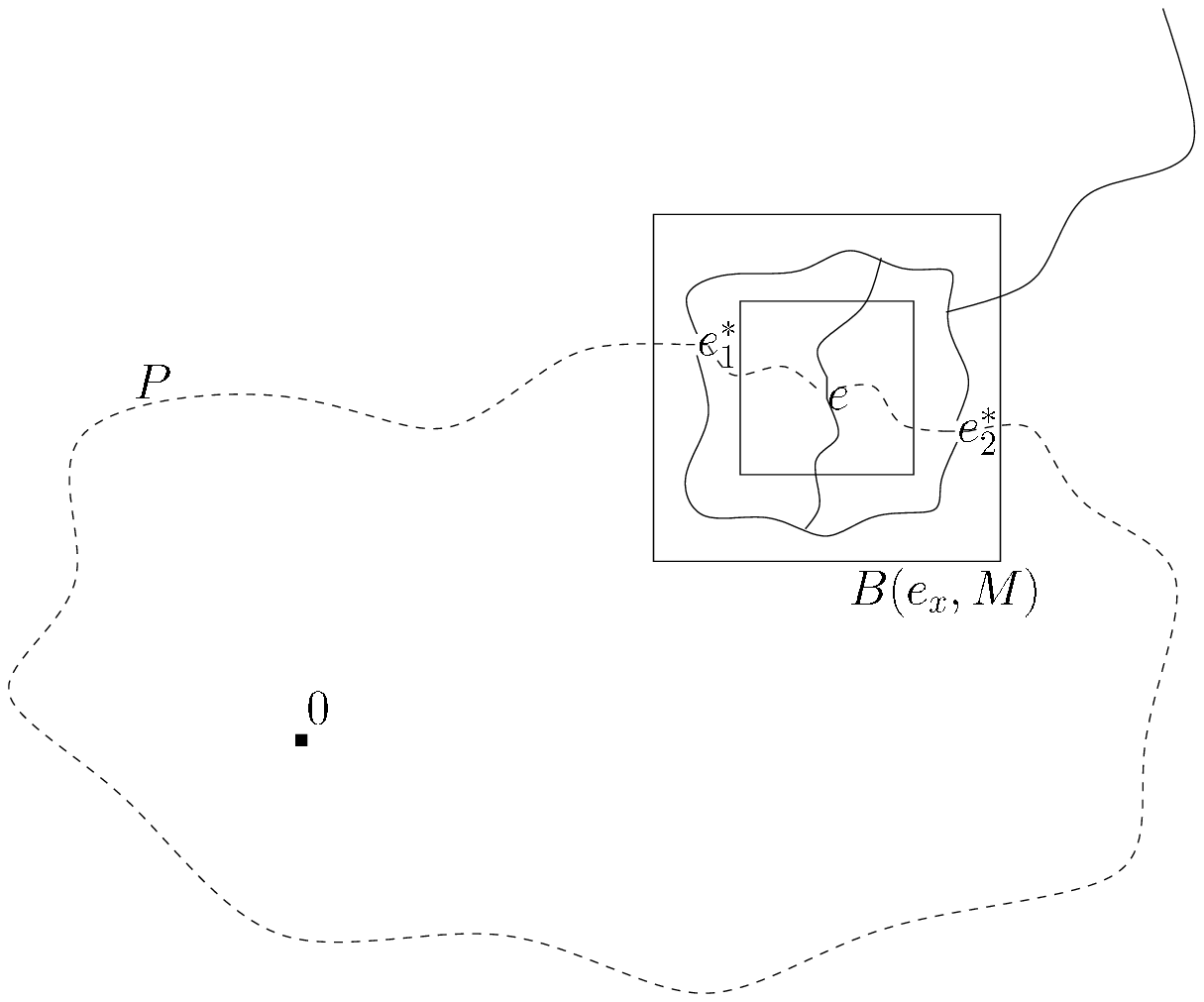}}
\end{center}
\caption{The edge $e$ is connected to the circuit $\theta_e{\cal C}$ by two $\tau_e$-open paths (the solid lines) and two $p_c$-closed paths (the dotted lines).  The outer dotted circuit represents the $\tau_e$-closed path $P$ and the circuit $\theta_e{\cal C}$ is connected to $\infty$ by a $\tau_e$-open path.  It is assumed that the invasion from the origin touches $\theta_e{\cal C}$ before it touches $P$.}
\label{fig2}
\end{figure}

We will denote by $e \leftrightarrow_{2,2,p_c} \theta_e{\cal C}$ the event that the first two events occur, we will denote by $\theta_e{\cal C} \leftrightarrow_{\tau_e} \infty$ the third event, and we will use the symbol $X({\cal C})$ for the fifth event.  See Figure \ref{fig2} for an illustration of the intersection of these events.  The term (\ref{tauconditioneq}) becomes
\begin{equation}
\label{bigeq}
{\E}\left[ {\P}(\theta_eE, H, \theta_eD({\cal C}), e \leftrightarrow_{2,2,p_c} \theta_e{\cal C}, \theta_e{\cal C} \leftrightarrow_{\tau_e} \infty, Q(\theta_e{\cal C})^c, \tau_e \in [p_c,p_n(1)), X({\cal C}) ~|~ \tau_e, \eta) \right].
\end{equation}
On the event $\left\{ e \leftrightarrow_{2,2,p_c} \theta_e{\cal C} \right\} \cap \left\{ \theta_e{\cal C} \leftrightarrow_{\tau_e} \infty \right\} \cap \tau_e \in [p_c,p_n(1))$, the event $H$ occurs if and only if there exists a $p_n(1)$-open circuit ${\cal C}_0$ enclosing the origin in $Ann(n/4,n/2)$ and either one of the following occur:

\[ {\cal C}_0 \stackrel{p_n(1)}\longleftrightarrow \theta_e{\cal C}, ~~\textrm{ or }~~ {\cal C}_0 \stackrel{p_n(1)}\longleftrightarrow \infty \textrm{ outside of } \theta_e{\cal C}. \]
Denote by $Y$ the event that such a circuit ${\cal C}_0$ exists and that either one of the above occur.  Note that $Y$ is measurable with respect to $\eta$.  The term (\ref{bigeq}) becomes
\[ {\E}\left[ {\P}(\theta_eE, Y, \theta_eD({\cal C}), e \leftrightarrow_{2,2,p_c} \theta_e{\cal C}, \theta_e{\cal C} \leftrightarrow_{\tau_e} \infty, Q(\theta_e {\cal C})^c, \tau_e \in [p_c,p_n(1)), X({\cal C}) ~|~ \tau_e, \eta) \right] \]
\begin{equation}
\label{decoupleeq}
= {\E}\left[ 1_Y 1_{\theta_eD({\cal C})} 1_{\theta_e{\cal C} \leftrightarrow_{\tau_e} \infty} 1_{\tau_e \in [p_c,p_n(1))} 1_{X({\cal C})} {\P} (\theta_eE, e \leftrightarrow_{2,2,p_c} \theta_e{\cal C}, Q(\theta_e{\cal C})^c ~|~ \tau_e, \eta) \right].
\end{equation}

We now inspect the inner conditional probability.  Clearly we have
\begin{equation}
\label{towardIICeq}
{\P}(\theta_eE, e \leftrightarrow_{2,2,p_c} \theta_e{\cal C}, Q(\theta_e{\cal C})^c ~|~ \tau_e, \eta) \leq {\P}(\theta_eE, e \leftrightarrow_{2,2,p_c}\theta_e{\cal C} ~|~ \tau_e,\eta)
\end{equation}
\[ \leq {\P}(\theta_eE, e \leftrightarrow_{2,2,p_c} \theta_e{\cal C}, Q(\theta_e{\cal C})^c ~|~ \tau_e, \eta) + {\P}(Q(\theta_e{\cal C})). \]
Using arguments similar to those that led to (\ref{Qconditioneq}), one can show that the same choice of $M$ and $N$ that will make (\ref{fourthQeq}) hold will also make

\[ \frac{1}{{\PP}(e \in {\cal O})} \sum_{{\cal C} \subset Ann(N,M)}{\E}\left[ 1_Y 1_{\theta_eD({\cal C})} 1_{\theta_e{\cal C} \leftrightarrow_{\tau_e} \infty} 1_{\tau_e \in [p_c,p_n(1))} 1_{X({\cal C})} {\P} (Q(\theta_e{\cal C})) \right] < \epsilon. \]
Therefore we conclude from (\ref{towardIICeq}) that ${\PP}(\theta_eE~|~e\in {\cal O})$ is within $4 \epsilon$ of 

\begin{equation}
\label{tempeq3}
\frac{1}{{\mathbb P}(e\in {\mathcal O})} \sum_{{\mathcal C}\subset Ann(N,M)} {\E}\left[ 1_Y 1_{\theta_eD({\cal C})} 1_{\theta_e{\cal C} \leftrightarrow_{\tau_e} \infty} 1_{\tau_e \in [p_c,p_n(1))} 1_{X({\cal C})} {\PP}(\theta_eE, e\leftrightarrow_{2,2,p_c} \theta_e{\cal C}~|~\tau_e, \eta) \right] .
\end{equation}

\bigskip
\textit{Step 6}.  Notice that since the events $\theta_eE$ and $e \leftrightarrow_{2,2,p_c} \theta_e {\cal C}$ do not depend on $\tau_e$ or on $\eta$, we have

\begin{equation}
\label{tempeq4}
{\P}(\theta_eE, e\leftrightarrow_{2,2,p_c}\theta_e{\cal C}~|~\tau_e, \eta) = {\P}_{cr}(E, 0 \leftrightarrow_{2,2} {\cal C})~\textrm{a.s.},
\end{equation}
where $0 \leftrightarrow_{2,2} {\cal C}$ denotes the event that the edge $e-e_x$ is connected to ${\cal C}$ by two open paths and the dual edge $(e-e_x)^*$ is connected to $\{(e_1-e_x)^*, (e_2-e_x)^*\}$ by two closed paths such that all of these connections occur inside ${\cal C}$ and the open and closed paths alternate.  The quantity ${\P}_{cr}(E | 0 \leftrightarrow_{2,2} {\cal C})$ from the right side of (\ref{tempeq4}) approaches $\nu^{2,2} (E)$ as long as $N \to \infty$ as $|e| \to \infty$ (this is a slight extension of Theorem~\ref{thmIIC}) so, assuming this growth on $N$, we have

\begin{equation}
\label{IICineq}
\frac{1}{(1+\epsilon)}{\P}_{cr}(E, 0 \leftrightarrow_{2,2} {\cal C}) \leq \nu^{2,2} (E) {\P}_{cr}(0 \leftrightarrow_{2,2} {\cal C})
\end{equation}

\[ \leq \frac{1}{(1-\epsilon)}{\P}_{cr}(E, 0 \leftrightarrow_{2,2} {\cal C}). \]

We will now show how to complete the proof; directly afterward, we will show how to make a correct choice of $M$ and $N$.  For convenience, let us define $R$ to be the term which comprises the entire line of  (\ref{tempeq3}) and let $S$ be the same term with the symbol $\theta_eE$ omitted.  Using the estimate from (\ref{IICineq}) and equation (\ref{tempeq4}), we get

\begin{equation}
\label{finalesteq}
(1-\epsilon)S \nu \leq R \leq (1+\epsilon)S \nu,
\end{equation}
where we write $\nu$ for $\nu^{2,2}(E)$.  By (\ref{tempeq3}) and the line of text above it, applied to the sure event ($E=\Omega$), we have

\[ 1-4\epsilon < S < 1+4\epsilon.\]
Combining this with (\ref{finalesteq}) gives

\[ (1-\epsilon)(1-4\epsilon) \nu < R < (1+\epsilon)(1+4\epsilon) \nu, \]
and the middle term is within $4\epsilon$ of ${\P}(\theta_eE ~|~ e \in {\cal O})$.  The proof is complete once we make a suitable choice of $M$ and $N$.

\bigskip
\textit{Step 7: Choice of $M$ and $N$}.  Recall, from (\ref{firstQeq}), that we need the inequality

\begin{equation}
\label{lasteq2}
C_8(C_* \log n)^2 {\P}(A_{N,M}^{2,3*}(p_n(1),p_c) ~|~ A_{N,M}^{2,2}(p_n(1),p_c)) < \epsilon
\end{equation}
to hold.  In addition, we need to satisfy (\ref{fourthQeq}).  Using the facts that ${\P}_{cr} (A_M^{2,2}) \geq C_9/M^2$ and

\begin{equation}
\label{lasteq1}
{\P}(A_{N,M}^{2,3*}(p_n(1),p_c) ~|~ A_{N,M}^{2,2}(p_n(1),p_c)) < C_{10} (\frac{N}{M})^{\beta}
\end{equation}
 for some $\beta>0$ (which is easily proved for what will be our choice of $M$ and $N$, and which we assume for the moment), the reader may check that a choice of

\[ N = \log n \textrm{ , } M = (\log n)^{2+2/\beta} \]
satisfies these two conditions for $n$ large.  The reason that this choice satisfies (\ref{fourthQeq}) is that $(\log n)^{\gamma} (p_n(1) - p_c) \to 0$ for any $\gamma$ (use (\ref{ineqKesten}) and the fact that the 4-arm exponent is strictly smaller than 2 (see, e.g., Section~6.4 in \cite{Werner})).

We now prove (\ref{lasteq1}).  Let $Q(M)$ be the event that there exists an edge in $B(M)$ which has weight in the interval $[p_c, p_n(1))$.  If $Q(M)$ does not occur then the event $A_{N,M}^{2,3*}(p_n(1),p_c)$ implies the event $A_{N,M}^{2,3}(p_n(1),p_c)$ (i.e. the same event but with all five paths disjoint).  Therefore, by Reimer's inequality,

\[ {\P}(A_{N,M}^{2,3*}(p_n(1),p_c)) \leq {\P}(A_{N,M}^{2,3}(p_n(1),p_c)) + {\P}(Q(M)) \]

\begin{equation}
\label{lasteq3}
\leq {\P}(A_{N,M}^{2,2}(p_n(1),p_c)){\P}_{cr}(A_{N,M}^{0,1}) + |B(M)|(p_n(1)-p_c),
\end{equation}
where $A_{N,M}^{0,1}$ is the event that $B(N)$ is connected to $\partial B(M)$ by a $p_c$-closed path.  Putting this estimate into (\ref{lasteq1}), the term on its left is at most

\[ {\P}_{cr}(A_{N,M}^{0,1}) + |B(M)|\frac{p_n(1)-p_c}{{\P}(A_{N,M}^{2,2}(p_n(1),p_c))}. \]
Using the fact that

\[ {\P}(A_{N,M}^{2,2}(p_n(1),p_c)) \geq \frac{{\P}(A_{N,M}^{2,3}(p_n(1),p_c))}{{\P}_{cr}(A_{N,M}^{0,1})} \geq \frac{C_{11}N^2}{M^2{\P}_{cr}(A_{N,M}^{0,1})}, \]
we see that the term (\ref{lasteq3}) is at most

\[ {\P}_{cr}(A_{N,M}^{0,1})\left[ 1+ \frac{C_{12}M^4}{N^2} (p_n(1)-p_c) \right] \leq 2{\P}_{cr}(A_{N,M}^{1,0}) \leq C_{13} \left( \frac{N}{M} \right) ^{\beta} \]
for some $\beta>0$, as we have chosen $N$ and $M$ on the order of $\log n$.  This shows (\ref{lasteq1}) and completes the proof.

\bigskip
\textbf{Acknowledgments.}  We would like to thank C. Newman for suggesting some of these problems.  We thank R. van den Berg and C. Newman for helpful discussions.  We also thank G. Pete for discussions related to arm-separation statements for multiple-armed IIC's.

\end{document}